\documentclass[12pt,reqno]{article}
\usepackage{amsmath}
\usepackage{amssymb}
\usepackage{amsthm}
\usepackage{mathrsfs}
\usepackage{latexsym}
\usepackage{color}
\usepackage{authblk}
\usepackage{bbm}
\usepackage{float}

\newtheorem{theorem}{Theorem}[section]
\newtheorem{lemma}{Lemma}[section]
\newtheorem{corollary}{Corollary}[section]
\newtheorem{remark}{Remark}[section]
\newtheorem{definition}{Definition}[section]
\newtheorem{proposition}{Proposition}[section]
\newtheorem{example}{Example}[section]

\numberwithin{equation}{section}
\newcommand{\bth}{\begin{theorem}}
\newcommand{\ethe}{\end{theorem}}

\newcommand{\bre}{\begin{remark}}
\newcommand{\ere}{\end{remark}}

\newcommand{\ble}{\begin{lemma}}
\newcommand{\ele}{\end{lemma}}
\newcommand{\bde}{\begin{definition}}
\newcommand{\ede}{\end{definition}}

\newcommand{\bco}{\begin{corollary}}
\newcommand{\eco}{\end{corollary}}

\newcommand{\bpr}{\begin{proposition}}
\newcommand{\epr}{\end{proposition}}

\newcommand{\bexer}{\begin{exercise}}
\newcommand{\eexer}{\end{exercise}}
\newcommand{\breh}{\begin{hint}}
\newcommand{\ereh}{\end{hint}}

\newcommand{\bexam}{\begin{example}}
\newcommand{\eexam}{\end{example}}

\newcommand{\bfi}{\begin{fig}}
\newcommand{\efi}{\end{fig}}

\newcommand{\beao}{\begin{eqnarray*}}
\newcommand{\eeao}{\end{eqnarray*}\noindent}

\newcommand{\beam}{\begin{eqnarray}}
\newcommand{\eeam}{\end{eqnarray}\noindent}

\begin{document}
\title{Time multidimensional Markov Renewal chains - An algebraic approach}

\author[1]{\small Leonidas Kordalis \thanks{\texttt{lkordali@math.uoa.gr}}}

\author[1,2]{Samis Trevezas\thanks{\texttt{strevezas@math.uoa.gr}}}

\affil[1]{\small Department of Mathematics, National and Kapodistrian University of Athens, Panepistimiopolis, 15784, Athens, Greece.}
\affil[2]{MICS laboratoty, CentraleSupélec - Université Paris-Saclay, 3 Rue Joliot Curie, Gif-sur-Yvette, 91190, France.}
\date{}
\maketitle
\begin{abstract}
In this study, a new extension of the Markov Renewal theory is introduced by allowing time to evolve in multiple dimensions. The resulting chains are referred to as multi-time Markov Renewal chains and since this extension is new, the state space is assumed to be finite to cover the theoretical framework of applications, where the possible number of states of a physical system is finite. The flexibility of Markov renewal theory is still present in multiple time dimensions by allowing the sojourn times in the different states of the system to be arbitrarily selected from a multidimensional distribution. The convolution product of multidimensional matrix sequences plays a particular role in the development of the theory and some of its algebraic properties are given and explored, paying particular attention to the existence, the representation and the computation  of the convolutional inverse. Some basic definitions and properties of this new class are given as well as the multi-time Markov renewal equations associated to the resulting processes. The practical implementation is achieved very efficiently through a novel adaptation of the Gauss-Jordan algorithm for multidimensional matrix sequences.
\end{abstract}

\maketitle
%\vspace{3mm}\begin{keyword}

\textit{Markov renewal chains; semi-Markov chains;
Multi-time Markov renewal chains; Matrix convolution Convolutional inverse Gauss-Jordan algorithm Fast Fourier Transform (FFT).}
\vspace{3mm}

\section{Introduction} 
Markov renewal chains and their associated semi-Markov chains have been studied extensively in their classical formulation and have also found many applications in Reliability (e.g., \cite{barbu2006empirical}, \cite{barbu2004discrete}, \cite{math9161997}, \cite{d2021computation}, \cite{limnios2012semi}) and seismology (e.g \cite{pertsinidou2017application}, \cite{votsi2010semi}, \cite{votsi2021hypotheses}).
 The theoretical studies concern the development of the appropriate probabilistic framework and the associated statistical modeling and inference techniques in order to use them in practice. 
The probabilistic modelling in discrete time has recently received a lot of attention and there is a growing interest on the subject (\cite{barbu2009semi}). 
In (\cite{anselone1960ergodic}, \cite[]{girardin2006entropy}, \cite{howard1971dynamic}, \cite[]{moore1968estimation}, \cite{pyke1964limit}, \cite{Trev2}, \cite{cinlar}) the authors developed some aspects of the  theory of (discrete and continuous time) Markov renewal processes  such as  the asymptotic behavior of concrete classes of semi-Markov chains, central limit theorems for additive functionals of Markov renewal chains under several assumptions and the associated statistical inference using maximum likelihood and studying the asymptotic properties of the resulting estimators.

Renewal models in the time plane were introduced by J.J Hunter (\cite{hunter1974renewal}, \cite{hunter1974renewal1}, \cite[]{hunter1977renewal}) who defined, solved renewal equations in two variables  and studied the asymptotic behavior of a time-extended renewal model. In particular, Hunter obtained the  asymptotic joint distribution of the marginal counting processes, a strong law of large numbers for the bivariate counting process and gave the form of the associated renewal equations using Laplace transforms and convolution techniques for two-dimensional lifetimes. The main field of application in reliability concerns failure and repair processes, where two time variables naturally arise, such as operational and calendar time, or cumulative usage and load cycles (\cite{insua2020advances}, \cite{jack2009repair}, \cite{lai2006stochastic}, \cite{yang2001bivariate}). 

This paper aims to present a new extension of the well-known Markov renewal and semi-Markov chains, where the notion of time can be multidimensional. Our motivation is both theoretical and practical. On one hand, there is an important gap in the development of a general theory in the multi-time framework and on the other hand these theoretical results can serve as a guideline for practitioners to apply these models for several applications of interest.

A typical example from reliability theory where the ergodic behavior of the system plays a key role, involves multi-component repairable systems operating under variable load or environmental conditions. In such systems, transitions between states like operational, degraded, or failed are governed by multiple, coexisting time dimensions (e.g., calendar time, cumulative operating time, or stress cycles). Employing a multi-time Markov renewal structure allows for more accurate estimation of long-term reliability measures, such as steady-state availability, mean time to failure, and failure rate functions, under complex operational profiles (\cite{kovalenko1997}, \cite{limnios2012semi}).

D'Amico  (\cite{d2011age}, \cite{d2012weighted}, \cite{d2013wind}) presented a non-homogeneous age-usage semi-Markov model, which is mentioned as an indexed semi-Markov process, using a three dimensional process. However, in their approach, the notion of time is still one-dimensional. In general, time extension could be advantageous in situations where the semi-Markov property (or/and time homogeneity) doesn't hold in any of many possible definitions of time, but it can be recovered by tracking together all these time-like characteristics. Other situations, where a multi-time description could be beneficial include systems which allow zero-time events (batch arrivals, discretization effects).

Another motivating application comes from plant growth modeling and agriculture. Biological development processes often depend on multiple, concurrent time factors. For example, in agriculture, state transitions between growth stages (e.g., germination, leaf development, flowering, fruiting) are influenced by chronological time, accumulated thermal time (growing degree days), and photoperiod exposure. Capturing these dependencies within a stochastic, state-dependent framework fits naturally into the multi-time Markov renewal modeling approach, enabling the analysis of organogenesis, phenological progression, and resource allocation dynamics (\cite{kang2008}, \cite{morrison2008}).
From a methodological standpoint, a central mathematical tool in our development is the convolution product of multidimensional matrix sequences. The convolutional structure plays a fundamental role in describing the state evolution and solving the governing renewal equations. Furthermore, we explore key algebraic properties of the convolution operator, with particular emphasis on the existence, representation, and efficient computation of the convolutional inverse.

These computational challenges are addressed by proposing and analyzing FFT-based convolution methods \cite{d3ea2d52-5ab2-3128-8b80-efb85267295d}, along with a novel adaptation of the Gauss–Jordan algorithm for multidimensional matrix sequences. These tools enable the practical resolution of multi-time Markov renewal equations, making the proposed model applicable even for high-dimensional or large-state-space systems.

The structure of the paper is as follows. Section 2 introduces the algebraic framework and properties of multidimensional convolutions. Section 3 presents the formulation and properties of multi-time Markov renewal chains. Section 4 illustrates computational aspects and practical applications, followed by a discussion section. Finally, the appendices provide additional technical details related to the implementation of FFT-based convolution techniques.  
\section{Discrete multi-time convolution product} 

In this section we introduce the necessary notations and we study basic properties of the multidimensional discrete time convolution of real and matrix valued functions using algebraic techniques. Special mention is given 
to the existence, the representation and the computation of the convolutional inverse, which is necessary to obtain unique solutions for Markov renewal equations.

%In order to develop the theory, we need first to study some properties of the multidimensional discrete  time convolution of real and matrix valued functions using algebraic techniques.   

\subsection{Preliminaries}
 Let $E=\{1,2,\dots,s\}$ be a finite set, $\mathcal{M}_{s}:=\mathbb{R}^{s\times s}$
 the set of all real matrices on $E \times E$ and  $\mathcal{M}_{s}(\mathbb{N}^{d})$ the set of all $\mathcal{M}_{s}$-valued functions defined on $\mathbb{N}^{d}$.
 An element $A \in \mathcal{M}_{s}(\mathbb{N}^{d})$ 
 %we write $A:=\left(A(k_{1:d}); k_{1:d} \in \mathbb{N}^{d}\right)$,
 corresponds to a multidimensional sequence of matrices,
 where for fixed $k_{1:d} \in \mathbb{N}^{d}$, the matrix $A(k_{1:d})=\left(a_{ij}(k_{1:d})\right)_{i,j\in E}$.
 Notice also that $A$ can be represented equivalently
 as a matrix of real valued $d$-dimensional sequences, so $A=(a_{ij})_{i,j\in E}$, where $a_{ij}=a_{ij}(k_{1:d})_{k_{1:d}\in \mathbb{N}^{d}}\in R := \mathbb{R}^{\mathbb{N}^d}$. In this sense $A\in R^{s\times s}$, thus it is also a matrix with elements in a commutative ring. 
  We also reserve the notation $0_d$ for the null vector on $\mathbb{R}^{d}$, as well as $I_s$ and ${O}_s$ for the identity and the null matrix on $\mathcal{M}_{s}$ respectively. Additionally, the space $\mathbb{N}^d$ is equipped with the partial ordering
$k\leq l$, if and only if $k_u\leq l_u$ for all $1\leq u \leq d$, and we write $k < l$ if strict inequality
holds for at least one of its components.

\begin{definition} \label{def:convolution}
 The discrete $d$-dimensional \textit{convolution product} is a binary operation on $\mathcal{M}_{s}(\mathbb{N}^{d})$, where 
 for each $A,B \in \mathcal{M}_{s}(\mathbb{N}^{d})$, the element $A*B\in \mathcal{M}_{s}(\mathbb{N}^{d})$ is defined by 
\begin{equation}\label{def:convolution1}
[A*B]\left(k_{1:d}\right):= \sum_{l+l'= k_{1:d}}A(l) B(l'), \quad k_{1:d} \in \mathbb{N}^{d}.
\end{equation}
\end{definition} 
\begin{remark}\label{rem:convolution} In the above definition the emphasis is given on the representation
of $A*B$ as a multidimensional sequence of matrices. Nevertheless, it is also useful to interpret 
$A*B$ as a matrix of multidimensional sequences. Indeed, notice by (\ref{def:convolution1}) that 
		\begin{equation*}
        [A*B]_{ij}\left(k_{1:d}\right)= \sum_{r \in E}\sum_{l+l'= k_{1:d}}
		\alpha_{ir}(l)\beta_{rj}(l') , \quad i,j \in E ,\, k_{1:d} \in \mathbb{N}^{d},
    \end{equation*}
thus $A*B=\left([A*B]_{ij}\right)_{i,j\in E}$, where $
    [A*B]_{ij}=\sum_{r\in E}\alpha_{ir}*\beta_{rj}$.

In this way, the convolution of two elements in $\mathcal{M}_s(\mathbb{N}^{d})$ corresponds to the usual matrix multiplication in $R^{s\times s}$, with $R$-product
the convolution of real valued $d$-dimensional sequences.
\end{remark}
\begin{remark}
The computational complexity of the convolution product is an important practical issue to be addressed.
    Using the exact formula to compute $[A*B](l_{1:d})$ for all $d$-tuples of indices $0\leq l_{1:d}\leq k_{1:d}$ 
    %requires  $k_{1}\cdots k_{d}$ matrix products, each requiring $s^{3}$ operations,  and $(k_{1}-1)\cdots (k_d-1)$ matrix additions, 
    has $\mathcal{O}(s^{3} \,k^{2}_{1}\cdots k^{2}_{d})$ complexity. An effective way to reduce this complexity is to use the Fast Fourier Transform (FFT)[\cite{d3ea2d52-5ab2-3128-8b80-efb85267295d}] for matrix sequences.
    
    In the case of real sequences $(s=1)$ with $\mathcal{O}(k^{2}_{1}\cdots k^{2}_{d})$ operations, it is reduced to $\mathcal{O}\left( \prod_{i=1}^{d}k_{i}\cdot \log_{2}(k_i) \right)$. It is well known that the convolution product of real sequences can be computed much faster in this manner. Thus, it is natural to extend this idea
for matrix sequences (see Appendix A).
%by considering the Fourier transform of a sequence $A \in \mathcal{M}_{s}(\mathbb{N}^{d})$ is considered next.  
   \end{remark} 
    
An important role in the development of the theory will be played by the time $d$-dimensional unitary function $\mathbbm{1}$, where $\mathbbm{1}(k_{1:d})=1$, for all $k_{1:d}\in \mathbb{N}^{d}$. 
%For $d=1$, we reserve also the symbol $\mathbf{1}$.
Related to this function are also the column vector and the matrix $d$-dimensional sequences
$\mathbbm{1}_s:=(\mathbbm{1},\ldots,\mathbbm{1})^{\top}$ and
$\mathrm{dg}(\mathbbm{1}_s):=\mathrm{diag}\left\{\mathbbm{1}_s\right\}$ respectively. Furthermore, $\mathrm{dg}(\mathbbm{1}_s)$ corresponds to the component-wise summation operator, i.e.
\begin{equation*}
    [\mathrm{dg}(\mathbbm{1}_s)*A](k_{1:d})=\sum_{l=0_{d}}^{k_{1:d}}A(l),\quad k_{1:d}\in \mathbb{N}^{d}.
\end{equation*}
Additionally, the symbol $\mathbbm{1}_{s\times s}$ is reserved for the matrix sequence with components $\mathbbm{1}$. 
Finally, we define the matrix sequence $\mathbbm{I}_s$ given by
\begin{equation*}
     \mathbbm{I}_s(k_{1:d}) = \left\{ \begin{array}{ll}
        I_{s}, & \textrm{if $k_{1:d}=0_d$,}\\
       O_s, & \textrm{otherwise}.\end{array} \right. 
\end{equation*}
 %In the following Proposition we give some algebraic properties of the convolution product of matrix valued functions, that can be verified easily.

Below, we introduce the  convolutional powers of a matrix-valued function.
\begin{definition}\label{n-fold-m-1d}
Let $A\, \in \mathcal{M}_s(\mathbb{N}^{d})$ be a matrix-valued function and $n\in \mathbb{N}$. The \textit{$n$-fold convolution} of $A$, denoted by $A^{(n)}$, is defined as the matrix-valued function given by :
\begin{eqnarray*}
    A^{(0)} &:=& \mathbbm{I}_s \\
     A^{(n)}&:=& A*A^{(n-1)}, \quad n \geq 1. 
\end{eqnarray*}
From the above definition we can get directly that for all $n\geq 0$ and $k_{1:d}\in \mathbb{N}^{d}$,
\begin{equation}\label{gen-elem-n-fold-matrix}
    A^{(n)}(k_{1:d})=\sum_{\substack{l_{1},\ldots,l_{n}\in \mathbb{N}^{d} \\ l_{1}+\ldots+l_{n}=k_{1:d}}}A(l_{1})\cdots A(l_{n}).
\end{equation}
\end{definition}

%\begin{definition}\label{def-inv}
%Let $A \in \mathcal{M}_s(\mathbb{N}^{d})$. If there exists a matrix valued %function $B \in \mathcal{M}_s(\mathbb{N}^{d})$ such that
%\begin{equation*}
 %   A*B=e_0,
%\end{equation*}
%then $B$ is called the right convolutional inverse of $A$ (right inverse of $A$ %in the convolution sense) and it is denoted by $A^{(-1)}_{r}$.
%If there exists a matrix valued function $C \in \mathcal{M}_s(\mathbb{N}^{d})$ %such that
%\begin{equation}\label{sat-inv-r}
 %   C*A=e_0,
%\end{equation}
%then $C$ is called the left convolutional inverse of $A$ (left inverse of $A$ in %the convolution sense) and it is denoted by $A^{(-1)}_{l}$. 
%\end{definition}

%and it is called the convolutional inverse of $A$ and denoted by $A^{(-1)}$.

%The inverse of $A$ does not always exist. For example, denote a $A\in \mathcal{M}_s(\mathbb{N}^{d}) $, where $A(0_d)=\mathbb{O}_s$.  If $A$ had an inverse then we would derive   the following:
%\begin{equation*}
 %   \mathbb{O}_s=B(0_d)A(0_d)=[B*A](0_d)=e_0(0_d)=I_{s}.
%\end{equation*}
%Hence, the convolution inverse of a sequence of matrices  exists under certain conditions which are given in the sequel.

\subsection{A representation of the convolutional inverse}

In the following lemma we give a useful property of the convolutional powers of a matrix-valued function $A$, when the condition $A(0_d)=O_s$ is satisfied. This condition simplifies the computation of the convolutional inverse of a matrix-valued function, whenever it exists.
\begin{lemma} \label{n<k1} Let $A\in \mathcal{M}_s(\mathbb{N}^{d})$ be a matrix-valued function with $A(0_d)=O_s$. Then, for any choice $k_{1:d}\in \mathbb{N}^d$, we have $A^{(n)}(k_{1:d})=O_s$, for all $n> k_1+\dots + k_d$.
\end{lemma}

\begin{proof}
If $n>k_1+\dots +k_d$, then by (\ref{gen-elem-n-fold-matrix}) each product of $n$ terms which appears in the summand necessarily includes an $i$ with $l_{i}=0_d$. But, since $A(0_d)=O_s$ we infer that $A^{(n)}(k_{1:d})=O_s$.
\end{proof}
\begin{remark}\label{Rem-0e}
As an immediate consequence of the above lemma we get that for any fixed $k_{1:d}$, the matrix sequence  $A^{(n)}(k_{1:d})$ is eventually null. Besides,
since each term $\sum_{n}A^{(n)}(k_{1:d})$ is actually a finite sum, the matrix sequence given by the series $\sum_{n}A^{(n)}$ is well defined. 
\end{remark}
\begin{theorem}\label{algebraicp1} The $\mathbb{R}$-vector space $(\mathcal{M}_s(\mathbb{N}^{d}),+,\cdot)$, equipped with the bilinear product of the convolution operator $*$ given by Definition \ref{def:convolution}, is a unital associative $\mathbb{R}$-algebra, is isomorphic with the set of all $s\times s$ generating functions,
  it possesses $\mathbbm{I}_s$ as its unique identity element and it is commutative iff $s=1$. Furthermore, for $A\in \mathcal{M}_{s}(\mathbb{N}^{d})$, if $A(0_d)$ is singular, then $A$ has no left or right inverse, otherwise if $A(0_d)$ is non-singular, then $A$ is invertible with a common left and right inverse. In particular, the convolutional inverse $A^{(-1)}$ of $A$, whenever it  exists, is  given by 
\begin{equation}\label{form-matrix-inv-unit}
     A^{(-1)}=  \left[A(0_d)\right]^{-1}\sum_{n=0}^{\infty}(\mathbbm{I}_s-A_0)^{(n)}=\left[\sum_{n=0}^{\infty}(\mathbbm{I}_s-A^{*}_0)^{(n)}\right]\left[A(0_d)\right]^{-1},
 \end{equation}
 where
 \begin{eqnarray}
     A_{0}(k_{1:d})&=&\label{A0}A(k_{1:d})\left[A(0_d)\right]^{-1},\quad k_{1:d}\, \in \mathbb{N}^{d},\\
     A^{*}_{0}(k_{1:d})&=&\nonumber \left[A(0_d)\right]^{-1}A(k_{1:d}),\quad k_{1:d}\, \in \mathbb{N}^{d}.
 \end{eqnarray}
 Additionally, each term $A^{(-1)}(k_{1:d})$ in (\ref{form-matrix-inv-unit}) can be represented as a finite sum and is given by
 \begin{equation}\label{form-matrix-inv-unit-t}
         A^{(-1)}(k_{1:d})=  \left[A(0_d)\right]^{-1}\left[\sum_{n=0}^{k_1+\ldots +k_d}(\mathbbm{I}_s-A_0)^{(n)}(k_{1:d})\right]=\left[\sum_{n=0}^{k_1+\ldots +k_d}(\mathbbm{I}_s-A^{*}_0)^{(n)}(k_{1:d})\right]\left[A(0_d)\right]^{-1}.
    \end{equation}  
\end{theorem}

\begin{proof}
The verification that $\mathcal{M}_s(\mathbb{N}^{d})$ equipped with the above operations is a unital associative $\mathbb{R}$-algebra having $\mathbbm{I}_s$ as its unique identity element is immediate. Now, notice that the non-singularity of $A(0_d)$ is a necessary condition for the existence of a left or a right inverse of $A$, since by setting $k_{1:d}=0_d$, the equations $B(0_d)A(0_d)=\mathbbm{I}_s(0_d)=I_s$ or $A(0_d)B(0_d)=I_s$ should hold respectively. Next, we can prove directly that the above condition is also sufficient for the existence of both a left and a right inverse by verifying that the matrix functions given in the right side of the first and the second equality of equation (\ref{form-matrix-inv-unit}) satisfy $B*A=\mathbbm{I}_s$ and $A*B=\mathbbm{I}_s $ respectively. Since $(\mathcal{M}_s(\mathbb{N}^{d}),*)$ is a monoid, if $A$ has both a left and a right inverse, then they necessarily coincide and a unique inverse $A^{(-1)}$ exists. The verification procedure is similar for both cases, so we only deal with the existence of a left inverse.
 Additionally, without loss of generality it suffices to prove that
 \begin{equation}\label{st-Al-inv}
  \left(\sum_{n=0}^{\infty}(\mathbbm{I}_s-A)^{(n)}\right) * A=\mathbbm{I}_s, \quad \textrm{if} \ A(0_d)=I_s.
 \end{equation}
  Indeed, if the latter holds, then the expression of $A^{(-1)}$ in the right side of the first equality of (\ref{form-matrix-inv-unit}) follows from the fact that $A(k_{1:d})=A_0(k_{1:d})A(0_d)$, where $A_0(k_{1:d})$ is given by (\ref{A0}), and $A_0$ satisfies (\ref{st-Al-inv}). Let us now proceed to the proof of (\ref{st-Al-inv}). Since $[\mathbbm{I}_s-A](0_d)=O_{s}$, by Remark (\ref{Rem-0e}) the series below is well defined and 
\begin{eqnarray*}
   \left(\sum_{n=0}^{\infty}(\mathbbm{I}_s-A)^{(n)}\right) * A &=& \left(\sum_{n=0}^{\infty}(\mathbbm{I}_s-A)^{(n)}\right) *[\mathbbm{I}_s-(\mathbbm{I}_s- A)]\\
   &=&\sum_{n=0}^{\infty}(\mathbbm{I}_s-A)^{(n)}\ - \sum_{n=0}^{\infty}(\mathbbm{I}_s-A)^{(n+1)}\\
   &=& (\mathbbm{I}_s-A)^{(0)}=\mathbbm{I}_s.
 \end{eqnarray*}
 Additionally, by Lemma \ref{n<k1} we get that (\ref{form-matrix-inv-unit-t}) holds.
\end{proof} 

\begin{remark}
Theorem \ref{algebraicp1} is a unified representation of the convolutional inverse for any class of sequences (real or matrix-valued) in several  variables.
In particular, if $f$ is a real sequence defined on $\mathbb{N}^{d}$ with $f(0_d)\neq 0$, then the corresponding convolutional inverse exists and is given by
\begin{equation}\label{form-real-inv-unit-t}
    f^{(-1)}(k_{1:d})=\frac{1}{f(0_d)}\left[\sum_{n=0}^{k_1+\ldots+k_d}\left(\mathbb{I}_1-\frac{1}{f(0_d)}f\right)^{(n)}(k_{1:d})\right],
\end{equation}
where $\mathbb{I}_1(0_d)=1$ and $0$ for $k_{1:d}\neq 0_d$.

\end{remark}

\subsection{Inversion with the classical recurrence}

In many cases, the form of a sequence is complex, making it difficult to compute
$A^{(-1)}$ (or $f^{(-1)}$), directly using expression such as   (\ref{form-matrix-inv-unit-t}), or (\ref{form-real-inv-unit-t}), However, it is possible to apply the following recursive method, which provides a practical alternative for such cases and will be referred to as the \textit{classical recurrence}. 
A proof for the one-variable case is given in \cite{barbu2009semi} and the extension to multiple time dimensions is straightforward. 
\begin{proposition} Let $A \in \mathcal{M}_{s}(\mathbb{N}^{d})$ be a convolutionally  invertible sequence. Its convolutional inverse $A^{(-1)}$ can be  computed recursively as follows:
   \[  A^{(-1)}(k_{1:d})= \left\{ \begin{array}{ll}
         \left[A(0_d)\right]^{-1},& \mbox{if $k_{1:d}=0_d,$}\\
         &\\
        \displaystyle -\,\left[A(0_{d})\right]^{-1}\left[\sum_{l< k_{1:d}}A^{(-1)}(l)A(k_{1:d}-l) \right], & \mbox{ otherwise}.\end{array} \right. \]  
\end{proposition}
This formulation enables the computation of each component of $A^{(-1)}$, iteratively, using previously computed terms and an additional multiplication with the inverse of $A(0_d)$.
The resulting complexity of the above method is $\mathcal{O}\left(s^{3} \prod_{i=1}^{d}k^{2}_i\right)$. Therefore, if we need to compute the convolutional inverse up to time $(2^{N},\ldots,2^{N})$, the complexity becomes to  $\mathcal{O}\left(s^{3}\cdot 4^{N\cdot d}\right)$.

\subsection{Newton's inversion method}
To reduce computational complexity, we employ matrix power series techniques. In the case of real sequences, the convolutional inverse can be efficiently computed using \textit{Newton's inversion method} \cite{Newton}.  Given a power series $$a(x_{1:d})=\sum_{k_{1:d}} f(k_{1:d}) x^{k_1}_1\cdots x^{k_d}_{d},$$ with $f(0_d)\neq 0$, we seek another power series $$b(x_{1:d})=\sum_{k_{1:d}} g(k_{1:d}) x^{k_1}_1\cdots x^{k_d}_{d}$$ such that \begin{equation}\label{identity}
    a(x_{1:d}) \cdot b(x_{1:d})=1,
\end{equation}
i.e., $b(x_{1:d})$ is the multiplicative inverse of $a(x_{1:d})$ or using the isomorphism in Theorem \ref{algebraicp1}, the coefficients $g$ form  the convolutional inverse of $f$.  
The performance of this computation can be significantly improved using Newton's iteration to achieve high precision very quickly.
First, we truncate (\ref{identity}) to a finite degree and take the approximation $b_{N}$ of $b$ such that %
\begin{equation*}
    a\cdot b_{N}= 1 \mod \left<x_{1},\ldots, x_d\right>^{2^N}.
\end{equation*}
The Newton's iterative method for approximating $b$ is given by:
\begin{equation}\label{Newton1}
  b_{N+1}  = b_{N}+b_{N}\cdot (2-a\cdot b_{N}) \mod \left<x_{1},\ldots, x_d\right>^{2^{N+1}},
\end{equation}
The proof of (\ref{Newton1}) is deferred to the Appendix \ref{1-2-3}. 
The  corresponding complexity up to time  $(2^{N},\ldots,2^{N})$, when using FFT, is reduced to $\mathcal{O}\left( N^{d}\cdot 2^{d\cdot N}\right)$. 

By adapting $(\ref{Newton1})$ to the matrix-valued case, we get similarly
\begin{equation}\label{Newton2}
  B_{N+1}  = B_{N}+B_{N}\cdot (2-A\cdot B_{N}) \mod \left<x_{1},\ldots, x_d\right>^{2^{N+1}},
\end{equation}
with initial condition $B_{0_d}=[A(0_d)]^{-1}$.
This formulation allows us to double the degree of approximation at each iteration, making the method particularly suitable for large-scale computations. As a result, the computational complexity for calculating the convolutional inverse is reduced to $\mathcal{O}\left(s^{3}\prod_{i=1}^{d}k_{i}\log_{2}(k_i)\right).$ As a result, the overall computational complexity of this procedure up to time $(2^{N},\ldots,2^{N})$ is $\mathcal{O}\left(s^{3}\cdot N^{d}\cdot 2^{d\cdot N}\right)$.

\subsection{Inversion with the Gauss-Jordan algorithm}
An alternative and often more practical method for computing the convolutional inverse is a convolution-based \textit{Gauss-Jordan algorithm}. To implement this, we first introduce the corresponding elementary row operations for multidimensional matrix sequences:
 \begin{itemize}
     \item  Swapping any two rows (vector sequences).
     \item Convolution of a row with a convolutionally invertible real sequence.
     \item  Replacing a row with the sum of itself and another row convolved with a convolutionally invertible real sequence. 
 \end{itemize}
   Using these operations, we introduce the following class of elementary matrix sequences.
    \begin{definition}
        A sequence  $E \in \mathcal{M}_{s}(\mathbb{N}^{d})$ is called \textit{elementary} if it is obtained by  applying a single row operation to the identity  sequence  $\mathbbm{I}_s$. Specifically:
        \begin{itemize}
            \item  Permutation matrix sequences produced by row switching operations on $\mathbbm{I}_s$, denoted by $E_{ij}$.
            \item Scalling matrix sequences, where  the $i$-th row  of $\mathbbm{I}_s$ is convolved with  a real sequence $\alpha$, denoted by $D_{i}[\alpha]$.
            \item  Row addition matrix sequences, where the $j$-th row is updated by adding the $i$-th row convolved with   a real sequence $\alpha$, denoted by $R_{ij}[\alpha]$.
        \end{itemize}

    \end{definition}
    \begin{lemma}
        
Let $A\in \mathcal{M}_{s}(\mathbb{N}^{d})$. Performing any  elementary row  operation on $A$  corresponds to  the convolution $E*A$, where $E$ is an appropriate  elementary matrix sequence  derived by $\mathbbm{I}_s$.
    \end{lemma}
    \begin{lemma}
        Any elementary matrix sequence is convolutionally invertible, and its inverse is also an  elementary matrix sequence. Specifically: \begin{itemize}
            \item [$\bullet$]$E^{(-1)}_{ij}=E_{ij}$,
                        \item [$\bullet$]$D^{(-1)}_{i}[\alpha]=D_{i}[\alpha^{(-1)}]$,
            \item [$\bullet$]$R^{(-1)}_{ij}[\alpha]=R_{ij}[-\alpha]$.

        \end{itemize}
    \end{lemma}

We now formally define row equivalence for matrix sequences.

    \begin{definition}
        Two sequences $A,B\in \mathcal{M}_{s}(\mathbb{N}^{d})$ are called \textit{row equivalent} if one can be obtained from  the other by  applying a finite sequence of elementary row operations.
    \end{definition}
\begin{proposition}
    Let $A,B\in \mathcal{M}_{s}(\mathbb{N}^{d})$. The sequences $A$ and $B$ are equivalent if there exists a finite $k\in \mathbb{N}$ and elementary sequences $C_{1},\ldots C_{k}$ such that
        $B=C_{k}*\ldots * C_{1}* A$.
\end{proposition}
Next, we define the convolutional analogue of the row echelon form.
\begin{definition} A sequence $A\in \mathcal{M}_{s}(\mathbb{N}^{d})$ is in \textit{convolutional row echelon form} if each of its component matrices is in row  echelon form. In particular, $A$ is said to be in \textit{reduced  convolutional row echelon form}  if $A(0_d)$ is in reduced echelon form and all other  matrices in the sequence  are zero.
    
\end{definition}
\begin{proposition}[Equivalence Relation]
The row equivalence relation defined on $\mathcal{M}_{s}(\mathbb{N}^{d})$ is an equivalence relation.
    
\end{proposition}
\begin{corollary} \label{help} For any $A\in \mathcal{M}_{s}(\mathbb{N}^{d})$ there exists an integer $k \in \mathbb{N}$ and a sequence of  elementary sequences of matrices $C_{1},\ldots, C_{k}$, such that the sequence $C_{k}*\ldots *C_{1}*A$ is in reduced convolutional  echelon form.
\end{corollary}
We now present an important characterization of convolutional invertibility for matrix sequences.
\begin{proposition} \label{Gausseq}
    For a sequence  $A\in \mathcal{M}_{s}(\mathbb{N}^{d})$, the following statements are equivalent.
    \begin{itemize}
        \item [$(i)$] $A$ has a convolutional inverse.
        \item [$(ii)$] $A$ is row equivalent to the identity sequence  $\mathbb{I}_s$.
        \item [$(iii)$] $A$ can be written as the convolution of a finite sequence of elementary matrix sequences.
    \end{itemize}
\end{proposition}
\begin{proof}
\

\begin{itemize}
    \item [$\bullet$] $(ii)\Longrightarrow (iii)$ and $(iii)\Longrightarrow (i)$ follow directly from the definitions and the fact that elementary sequences are convolutionally invertible.
    \item[$\bullet$]  For $(i)\Longrightarrow (ii)$, we apply  Corollary \ref{help} to bring $A$ into reduced convolutional echelon form. Denoting this form  as $B$, we write  $B=C_{k}*\ldots*C_{1}*A.$ Since,  each $C_i$ and $A$ are invertible, $B$ also has a convolutional inverse. By Theorem \ref{algebraicp1}, this implies that $B(0_d)$ cannot contain zero rows, meaning $B$ is equivalent to $\mathbb{I}_s$. Therefore, by transitivity, $A$ is equivalent to $\mathbb{I}_s$.
\end{itemize}

\end{proof}
 Proposition \ref{Gausseq} indicates the following equivalence:
\begin{equation*}
  \mathbb{I}_s=C_{k}*\ldots * C_{1}* A \ \ \Longleftrightarrow  \ \ A^{(-1)}=C_{k}*\ldots * C_{1}.
\end{equation*}
Therefore, we can compute $A^{(-1)}$, starting from $\mathbb{I}_s$,  by applying the same sequence of elementary matrix operations that transforms $A$ into $\mathbb{I}_s$.

Let us give an example. Consider the sequence $A\in \mathcal{M}_{2}(\mathbb{N}^{d})$, where  
$A=\begin{pmatrix}
   a & b \\
   c & d
\end{pmatrix}$.

Then, the convolutional inverse of $A$ can be computed through the following sequence of operations:

$\begin{pmatrix}
   a \ & \ b \\
   c\ & \ d
\end{pmatrix}$$\xrightarrow{D_{1}[\alpha^{(-1)}]}{}$ $\begin{pmatrix}
   \mathbb{I}_1 \ & \ b*\alpha^{(-1)} \\
   c & d
\end{pmatrix}$ $\xrightarrow{R_{12}[-c]}{}$  $\begin{pmatrix}
   \mathbb{I}_1\ & \ b*\alpha^{(-1)} \\
   0 \ &\ d-b*\alpha^{(-1)}*c
\end{pmatrix}$  $\xrightarrow{D_{2}\left[\left(d-b*\alpha^{(-1)}*c\right)^{(-1)}\right]}{{}}$

$\begin{pmatrix}
   \mathbb{I}_1 \ & \ b*\alpha^{(-1)} \\
   0 \ & \ \mathbb{I}_1
\end{pmatrix}$ $\xrightarrow{R_{21}[-b*\alpha^{(-1)}]}{}$ $\begin{pmatrix}
   \mathbb{I}_1\ &\ 0 \\
   0\ & \ \mathbb{I}_1
\end{pmatrix}$.

Based on this  process, the convolutional inverse of $A$ is given by
\begin{equation*}
    A^{(-1)}=R_{21}[-b*\alpha^{(-1)}]*D_{2}\left[\left(d-b*\alpha^{(-1)}*c\right)^{(-1)}\right]*R_{12}[-c]* D_{1}[\alpha^{(-1)}].
\end{equation*}

This procedure requires computing the convolution of the corresponding elementary matrix sequences and calculating the inverses $\alpha^{(-1)}$ and $ \left(d-b*\alpha^{(-1)}*c\right)^{(-1)}$. To accomplish this efficiently, we use FFT-based convolution techniques combined with the Newton's method for inverting real sequences. As a result, the overall computational complexity of this procedure is $\mathcal{O}\left(s^{3}\prod_{i=1}^{d}k_i \log_2(k_i)\right)$.

\section{Multi-time Markov renewal chains}\label{extension}

Consider a random system with finite state space $E =\{1,...,s\}$. We assume that the system is evolving with jumps in a d-dimensional time domain and a process $(J_n)$ records
the successively visited states. In this work we allow jumps at the same state, so renewal processes and
in particular multi-time renewal chains can be considered as a special case of multi-time Markov renewal chains for $s=1$.
The jump times are recorded by the d-dimensional process $(S_n)=(S_n^{[u]})_{1\leq u \leq d}$ with values in $\mathbb{N}^d$ which corresponds to the d-dimensional discrete time domain. We start by assuming that $S_0=0_d$, but this assumption could be relaxed if a delay is allowed.  In this study, we assume that when a jump takes place, then $S_n<S_{n+1}$ for all $n\in \mathbb{N}$ and we say that
$(S_n)$ strictly increases. Note that this does not necessarily hold true for all of its components, eventually incurring zero-time jumps for some of them.  
We also define $(X_n)=(X_n^{[u]})_{1\leq u \leq d}$ the sequence of sojourn (inter-jump) times, where 
by convention $X_0=0$ and $X_n=S_n - S_{n-1}$, for all $n\geq 1$, and the counting process of the number of jumps in the integer time grid $\prod_{i=1}^{d}[0;k_i]$ given by 
\begin{equation*}
	N(k_{1:d})	:= \sup \{n\in \mathbb{N}: S_n \leq k_{1:d}\}.
\end{equation*} 
Additionally, if we denote by $N^{[u]}$ the marginal jump counting process associated with the $u$-th time component, $1\leq u \leq d$, then it is straightforward to see that
        \begin{equation}\label{N_as_min}
            N(k_{1:d})=\min_{1\leq u\leq d}\{N^{[u]}(k_u)\},\ k_{1:d}\in \mathbb{N}^{d}.
        \end{equation}
A time multidimensional extension of the classical semi-markov kernel is defined as follows.
\begin{definition}
A matrix-valued function $q=\big(q_{ij}(k_{1:d})\big)\in \mathcal{M}_{s}(\mathbb{N}^d)$ is said to be a discrete d-dimensional \textit{multi-time semi-markov kernel} (mt-smk or dt-smk, if we want to refer to its dimension $d$ explicitly), if it satisfies the following two conditions:
\begin{itemize}
 \item[$(i)$] $q_{ij}(k_{1:d}) \geq  0,\ i,j \in E,\ k_{1:d} \in \mathbb{N}^d$, 
 %\item[$(ii)$] $q_{ij}(0_{1:d}) = 0, \ i,j \in E $, 
 \item[$(ii)$] $\displaystyle \sum_{j}\sum_{k_{1:d}} q_{ij}(k_{1:d}) =1 ,\ i \in E$. 
\end{itemize}
If additionally, the following condition holds
\begin{itemize}
 \item[$(iii)$] $q_{ij}(0_{d}) = 0, \ i,j \in E $,
\end{itemize}
then $d$-dimensional instantaneous transitions are not allowed. 
\end{definition} 

\begin{remark}
If we define the matrix-valued functions $q^{[u]}=(q^{[u]}(k_u))_{k\geq 0} \in \mathcal{M}_{s}(\mathbb{N})$,
 where 
 \begin{equation}\label{kernels}
 q^{[u]}_{ij}(k_u)=\sum_{k_{1:u-1}}\sum_{k_{u+1:d}}q_{ij}(k_{1:u-1},k_u,k_{u+1:d}),
 \end{equation}
 then one gets directly that for $1\leq u \leq d$, the matrix-valued function $q^{[u]}$ is a $1t$-smk, possibly allowing instantaneous transitions. The same holds true for any other simultaneous selection of $d_0<d$ time components, thus forming a $d_0 t$-smk.
\end{remark}
Some quantities of interest related directly to the semi-Markov kernel concern the \textit{cumulated semi-Markov kernel} $Q=\left(Q_{ij}(k_{1:d})\right)\in \mathcal{M}_{s}(\mathbb{N}^d)$:
\begin{equation}\label{cumulated_kernel}
      Q_{ij}(k_{1:d})=\mathbb{P}\left[X_{n+1}\leq k_{1:d},\,J_{n+1}=j \mid J_{n}=i\right] 
    \end{equation}
    and the \textit{survival semi-Markov kernel} $\overline{Q}=\left(\,\overline{Q}_{ij}(k_{1:d})\right)\in \mathcal{M}_{s}(\mathbb{N}^d)$:
    \begin{equation}\label{survival_kernel}
        \overline{Q}_{ij}(k_{1:d})=\mathbb{P}\left[X_{n+1}>k_{1:d},\,J_{n+1}=j \mid J_{n}=i\right].
    \end{equation}
Now, we can introduce the multi-time Markov renewal chains.
\begin{definition}\label{tb-MRC}
The chain $(J,S):=(J_{n},S_{n})_{n \in \mathbb{N}}$ is said to be a $d$-dimensional multi-time Markov-renewal chain (mt-MRC or dt-MRC, if we want to refer to its dimension $d$ explicitly) if for all $n \in \mathbb{N}$,
$i,j \in E$ and $k_{1:d} \in \mathbb{N}^d$, it satisfies a.s.
 \begin{equation}\label{eq:MRC}
    \mathbb{P}(J_{n+1}=j, S_{n+1}-S_{n}=k_{1:d} \,|\, J_{0:n},S_{0:n})
		= \mathbb{P}(J_{n+1}=j, S_{n+1}-S_{n}=k_{1:d}\,|\, J_{n}).
\end{equation}
Besides, if (\ref{eq:MRC}) is independent of $n$, then $(J,S)$ is said to be homogeneous and we set
  \begin{equation*}
    q_{ij}(k_{1:d}):=\mathbb{P}(J_{n+1}=j, S_{n+1}-S_{n}=k_{1:d}\,|\,  J_{n}=i).
\end{equation*}
The matrix-valued function $q=\big(q_{ij}(k_{1:d})\big)\in \mathcal{M}_{s}(\mathbb{N}^d)$ satisfies the properties of a dt-smk and it will be referred to as its associated semi-Markov kernel.
\end{definition}
From an mt-MRC $(J,S)$, we can deduce directly some Markov chains of interest.
\begin{proposition}
Let $(J,S)$ be a dt-MRC and $q$ its associated dt-smk. Then, the processes
$(J,S)$, $(J,X)$, $(J,S^{[u]})$ and $(J,X^{[u]})$, $u\in [1;d]$, are Markov chains with transition probabilities:
\begin{eqnarray*}
\mathbb{P}(J_{n+1}=j,\, S_{n+1}=s_{1:d} + k_{1:d} \,|\, J_{n}=i,\, S_{n}=s_{1:d}) & = &  q_{ij}(k_{1:d}), \\
\mathbb{P}(J_{n+1}=j,\, X_{n+1}= k_{1:d} \,|\, J_{n}=i,\, X_{n}=k_{1:d}^{'}) & = &  q_{ij}(k_{1:d}), \\
\mathbb{P}(J_{n+1}=j,\, S_{n+1}^{[u]}=s_{u}+k_{u} \,|\, J_{n}=i,\, S_{n}^{[u]}=s_u) & = &  q_{ij}^{[u]}(k_{u}), \\
\mathbb{P}(J_{n+1}=j,\, X_{n+1}^{[u]}=k_u \,|\, J_{n}=i,\, X_{n}^{[u]}= k_{u}^{'}) & = & q^{[u]}_{ij}(k_u), 
%\\
%\mathbb{P}(J_{n+1}=j | J_{n}=i) & = & p_{ij},
\end{eqnarray*}  
where $q^{[u]}_{ij}(k_u)$ is defined in (\ref{kernels}). The chain
 $(J,S^{[u]})$ is a 1t-MRC with associated smk $q^{[u]}$. 
The process $J$ is also a Markov chain, called the embedded Markov chain associated to the dt-MRC $(J,\,S)$, as well as to the 1t-MRC $(J,S^{[u]})$. Its transition probabilities are given by
\begin{equation}\label{eq:trans_prob}
p_{ij}:=\mathbb{P}(J_{n+1}=j \,|\, J_{n}=i) = \sum_{k_{1:d}} q_{ij}(k_{1:d}).
\end{equation}
\end{proposition}
Next, in order to study the probabilistic framework of an mt-MRC, we need to define two types of distributions, the sojourn time distributions in a given state and the corresponding conditional distributions given the next visited state. 
\begin{definition} \label{mt-MRC} 
Let $(J,S)$ be a $dt$--MRC. For any $i \in E$, $k_{1:d}\in \mathbb{N}^d$ and $k_{u}\in \mathbb{N}, \ u\in [1;d]$, we introduce the following functions related to the sojourn time distribution of $[X_{n+1}|J_n=i], \  n \in \mathbb{N}$:
\begin{itemize}
    \item[$(i)$] the joint probability mass function $h_i\in \mathcal{M}_1(\mathbb{N}^d)$:
    \begin{equation*}
        h_{i}(k_{1:d}):=\mathbb{P}(X_{n+1}=k_{1:d}\mid J_{n}=i)\quad \textrm{and} \quad h_{\!E}:=(h_{i})_{i\in E}\in \mathcal{M}_{s\times 1}(\mathbb{N}^d),
    \end{equation*}
    \item [$(ii)$] the joint (cumulative) distribution function $H_{i}\in \mathcal{M}_1(\mathbb{N}^d)$:
     \begin{equation*}
        H_{i}(k_{1:d}):=\mathbb{P}(X_{n+1}\leq k_{1:d}\mid J_{n}=i)\quad \textrm{and} \quad H_{\!E}:=(H_{i})_{i\in E}\in \mathcal{M}_{s\times 1}(\mathbb{N}^d),
     \end{equation*}
    \item[$(iii)$]  the joint complementary distribution function 
    $\widetilde{H}_{i}\in \mathcal{M}_1(\mathbb{N}^d)$:
    \begin{equation*}
        \widetilde{H}_{i}(k_{1:d}):=
       1- \mathbb{P}(X_{n+1}\leq k_{1:d}\mid J_{n}=i)\quad \textrm{and} \quad \widetilde{H}_{\!E} :=(\widetilde{H}_{i})_{i\in E}\in \mathcal{M}_{s\times 1}(\mathbb{N}^d),
    \end{equation*}
    and the associated matrix sequence $
      \widetilde{H}:= \mathrm{dg}\left( \widetilde{H}_{\!E}\right)\in \mathcal{M}_{s}(\mathbb{N}^d)$,
	\item [$(iv)$] the  joint survival function $\overline{H}_{i}\in \mathcal{M}_1(\mathbb{N}^d)$: 
	\begin{equation*}
	    \overline{H}_{i}(k_{1:d}):=\mathbb{P}\left(X_{n+1}>k_{1:d}\mid J_{n}=i\right)\quad \textrm{and} \quad \overline{H}_{\!E} :=(\overline{H}_{i})_{i\in E}\in \mathcal{M}_{s\times 1}(\mathbb{N}^d),
	\end{equation*}
     \item[$(v)$] the marginal probability mass function $h^{[u]}_{i}\in \mathcal{M}_1(\mathbb{N})$:
    \begin{equation*}
        h^{[u]}_{i}(k_u)=\mathbb{P}(X^{[u]}_{n+1}=k_{u}\mid J_{n}=i)\quad \textrm{and} \quad h^{[u]}_{\!E} :=(h^{[u]}_i)_{i\in E}\in \mathcal{M}_{s\times 1}(\mathbb{N}),
    \end{equation*}
    \item [$(vi)$] the marginal (cumulative) distribution function $H^{[u]}_{i}\in \mathcal{M}_1(\mathbb{N})$:
    \begin{equation*}
        H^{[u]}_{i}(k_u)=\mathbb{P}(X^{[u]}_{n+1}\leq k_{u}\mid J_{n}=i)\quad \textrm{and} \quad H^{[u]}_{\!E} :=(H_{i}^{[u]})_{i\in E}\in \mathcal{M}_{s\times 1}(\mathbb{N}),
    \end{equation*}
	\item [$(vii)$] the marginal survival function $\overline{H}^{[u]}_{i}=\widetilde{H}^{[u]}_{i}\in \mathcal{M}_1(\mathbb{N})$:
	\begin{equation*}
	    \overline{H}^{[u]}_{i}(k_u)= \mathbb{P}(X^{[u]}_{n+1}> k_{u}\mid J_{n}=i)\quad \textrm{and} \quad \overline{H}^{[u]}_{\!E} :=(\overline{H}_{i}^{[u]})_{i\in E}\in \mathcal{M}_{s\times 1}(\mathbb{N}).
	\end{equation*}
\end{itemize}
\end{definition}
\begin{remark}
 The sojourn time distribution function in a state $i\in E$ can also be formulated via the cumulated semi-Markov kernel $Q$ given by (\ref{cumulated_kernel}):
 \begin{equation*}
    H_{i}(k_{1:d})=\sum_{j \in E} Q_{ij}(k_{1:d}), \qquad k_{1:d}\in \mathbb{N}^d.
 \end{equation*}
\end{remark}
Now, we introduce important functions for the conditional distributions given the next visited state. 
\begin{definition} \label{mar-MRC} Let $(J,S)$ be a $dt$--MRC and $(J,S^{[u]})$ the associated marginal $1t$--MRCs, $u\in [1;d]$. For any $i,j \in E$, $k_{1:d}\in \mathbb{N}^d$ and $k_{u}\in \mathbb{N}, \ u\in [1;d]$, we introduce the following functions related to the conditional sojourn time distribution of $[X_{n+1}|J_n=i,J_{n+1}=j], \  n \in \mathbb{N}$:
\begin{itemize}
    \item[$(i)$] the conditional joint probability mass function $f_{ij}\in \mathcal{M}_1(\mathbb{N}^d)$:
    \begin{equation*}
        f_{ij}(k_{1:d}):=\mathbb{P}(X_{n+1}=k_{1:d}\mid J_{n}=i,J_{n+1}=j)\quad \textrm{and} \quad f:=(f_{ij})_{i,j\in E}\in \mathcal{M}_s(\mathbb{N}^d),
    \end{equation*}
    \item[$(ii)$]the conditional joint (cumulative) distribution function $F_{ij}\in \mathcal{M}_1(\mathbb{N}^d)$:
   \begin{equation*}
        F_{ij}(k_{1:d}):=\mathbb{P}(X_{n+1}\leq k_{1:d}\mid J_{n}=i,J_{n+1}=j)\quad \textrm{and} \quad F:=(F_{ij})_{i,j\in E}\in \mathcal{M}_s(\mathbb{N}^d),
    \end{equation*}
    \item[$(iii)$] the conditional marginal probability mass function $f^{[u]}_{ij}\in \mathcal{M}_1(\mathbb{N})$:
    \begin{equation*}
        f_{ij}^{[u]}(k_{u}):=\mathbb{P}(X^{[u]}_{n+1}=k_{u}\mid J_{n}=i,J_{n+1}=j)\quad \textrm{and} \quad
        f^{[u]}:=(f_{ij}^{[u]})_{i,j\in E}\in \mathcal{M}_s(\mathbb{N}),
    \end{equation*}
    \item[$(iv)$]the conditional marginal distribution function $F^{[u]}_{ij}\in \mathcal{M}_1(\mathbb{N})$:
    \begin{equation*}
        F_{ij}^{[u]}(k_{u}):=\mathbb{P}(X_{n+1}\leq k_{u}\mid J_{n}=i,J_{n+1}=j)\quad \textrm{and} \quad F^{[u]}:=(F_{ij}^{[u]})_{i,j\in E}\in \mathcal{M}_s(\mathbb{N}).
    \end{equation*}
\end{itemize}
\end{definition}
\begin{remark} \label{back-up}
The following relations hold for the quantities related to
\begin{itemize}
\item[(i)] the semi-Markov kernel:
\begin{eqnarray*}
    Q_{ij}=\mathbbm{1}*q_{ij}  \quad \text{and} \quad Q=\mathrm{dg}(\mathbbm{1}_s)*q,
    \end{eqnarray*}
\item[(ii)] the sojourn times distribution: 
\begin{eqnarray*}
     H_{i}=\mathbbm{1}*h_i
    \quad \textrm{and} \quad  H_{\!E}= \mathrm{dg}(\mathbbm{1}_s)*h_{\!E},
    \end{eqnarray*}
\item[$(iii)$] the conditional sojourn times distribution: 
    \begin{equation*}
        F_{ij}=\mathbbm{1}*f_{ij}\quad \textrm{and} \quad F=
        \mathrm{dg}(\mathbbm{1}_s)*f,
    \end{equation*}

    \item[$(iv)$]the marginal conditional distribution function of $X^{[u]}_{n+1}$, $n \in \mathbb{N}$, given $(J_n,J_{n+1})=(i,j)$:
    \begin{equation*}       F^{[u]}_{ij}=\mathbf{1}*f^{[u]}_{ij}\quad \textrm{and} \quad F^{[u]}:=(F_{ij}^{[u]})_{i,j\in E}=
        \mathrm{dg}(\mathbf{1}_s)*f^{[u]},
    \end{equation*}
    where $\mathbf{1}$ is the unitary function
    $\mathbbm{1}$ with $d=1$ and $\mathbf{1}_s$ the $s$-dimensional column vector of $\mathbf{1}$'s.
%        
%    \begin{equation*}
%        \widetilde{H}_{i}=\mathbbm{1}-H_{i}=\mathbbm{1}-\mathbbm{1}*h_{i},
%    \end{equation*}
    %
%      \begin{equation*}
%        %H^{[u]}_{i}=\mathbbm{1}*h_{i}^{[u]},
 %   \end{equation*}
\end{itemize}
\end{remark}
Some moments of interest related to the sojourn time distribution in a specific state of an mt-MRC and the associated marginal MRCs are given in the following definition.
\begin{definition}\label{moments}
 For a dt-MRC $(J,S)$ and the marginal MRCs $\{(J,S^{[u]})\}_{u=1}^{d}$, we introduce:
\begin{itemize}
\item[$(i)$] the $k$-th moment of the $u$-th sojourn time in state $i\in E$, $k\in \mathbb{N}^{*}$, $1\leq u \leq d$:
    \begin{equation*}   m^{[u,k]}_i:=\mathbb{E}\left[\left(X^{[u]}_{n+1}\right)^{k} \ \ \big|\ \  J_n=i\right],
    \end{equation*}
  while for the  first moments we reserve the symbol $m_{i}^{[u]}$, i.e., $m_{i}^{[u]}:=m_{i}^{[u,1]}$, 
    \item[$(ii)$] the mean sojourn time in state $i\in E$:
    \begin{equation*}
        m_i:=\left(m_{i}^{[u]}\right)_{i\in E}=\mathbb{E}\left[X_{n+1}\mid J_n=i\right],
    \end{equation*}
    \item[$(iii)$] the  mean sojourn time given the current and next state $(i,j)\in E^2$:
    \begin{equation*}
        m_{ij}:=\mathbb{E}\left[X_{n+1}\mid J_n=i,\,J_{n+1}=j\right],
    \end{equation*}
      \item[$(iv)$] the mean of the product $X^{[u]}_{n+1}\cdot X^{[v]}_{n+1}$ in  state $i\in E$, $1\leq u,v \leq d$:
    \begin{equation*}      m_i^{[u, v]}:=\mathbb{E}\left[X^{[u]}_{n+1}\cdot X^{[v]}_{n+1} \ \ \big|\ \  J_n=i\right],
    \end{equation*}
     \item[$(v)$] the conditional  mean of the product $X^{[u]}_{n+1}\cdot X^{[v]}_{n+1}$ given the current and next state $(i,j)\in E^2$, $1\leq u,v \leq d$:
    \begin{equation*}       
m_{ij}^{[u,v]}:=\mathbb{E}\left[X^{[u]}_{n+1}\cdot X^{[v]}_{n+1}\ \ \big|\ \ J_n=i,\, J_{n+1}=j\right],
    \end{equation*}
    \item [$(vi)$]the  covariance between $X_{n+1}^{[u]}$ and $X_{n+1}^{[v]}$ in state $i\in E$, $1\leq u,v \leq d$:
    \begin{equation*}
         c_{i}^{[u,v]}=m_i^{[u,v]}-m_i^{[u]} \cdot m_i^{[v]} ,
    \end{equation*}
    \item [$(vii)$] the conditional covariance between $X_{n+1}^{[u]}$ and $X_{n+1}^{[v]}$ given the current and next state $(i,j)\in E^2$, $1\leq u,v\leq d$ :
    \begin{equation*}
          c_{ij}^{[u,v]}= m_{ij}^{[u,v]}-m_{ij}^{[u]} \cdot m_{ij}^{[v]}.
    \end{equation*}\end{itemize}
\end{definition}

\begin{remark}\label{mean-sojourn}
 By \eqref{eq:trans_prob} and Definition \ref{moments}(ii)-(v), we get directly that
 \begin{eqnarray*}
     m_{i} &=&\sum_{j}p_{ij}m_{ij}, \\
     m_{i}^{[u,v]} &=&\sum_{j}p_{ij}m_{ij}^{[u,v]}.
 \end{eqnarray*}
\end{remark}
In order to compute the conditional covariance between two interjump times we need to compute the term $m_{i}^{[u,v]}=\mathbb{E}\left[X^{[u]}_{n+1}\cdot X^{[v]}_{n+1}\,|\,J_n=i\right]$. Since $X^{[u]}_{n+1}\cdot X^{[v]}_{n+1}$ is a nonnegative integer valued random variable, notice that
\begin{equation}\label{wanted}
   m_{i}^{[u,v]}=\sum_{k=0}^{\infty}\mathbb{P}\left[X^{[u]}_{n+1}\cdot X^{[v]}_{n+1}>k\,|\,J_n=i\right],
\end{equation}
%
%The term in the series of the right-hand in  (\ref{wanted}) can be approached by the properties of a  usual Markov Renewal chain as a consequence of  a mt-MRC. For the latter, we'll give a more  general result.
so the summand needs to be computed.
In the next proposition we give a more general result of independent interest.
%First, we denote by $S^{\Phi}_n=\sum_{k=1}^{n}\Phi_k$ the discrete time process given by $\Phi_n=\phi(X_n)$ and the following result holds.
%
  \begin{proposition}
        Let $(J,S)$ be a dt-MRC, $\phi:\mathbb{N}^{d}\to \mathbb{N}$  be a function and $\Phi_n:=\phi(X_n)$. If $S^{[\phi]}_n=\sum_{k=0}^{n}\Phi_k$, $n\geq 0$, then the chain $(J,S^{[\phi]})$  forms a Markov renewal chain with interjump times $(\Phi_{n})_{n\geq 1}$ and the associated semi-Markov kernel $q^{[\phi]}$ is given by:
        \begin{equation}\label{Phi-kernel}
            q^{[\phi]}_{ij}(k)=\sum_{\substack{v: \\ \phi(v)=k}}q_{ij}(v),\, i,\, j \in E,\, k\in \mathbb{N}.
        \end{equation}
  \end{proposition}

\begin{proof}
First, we get the following relations:
\begin{eqnarray*}
 \left\{S^{[\phi]}_{0}=u_{0},S^{[\phi]}_{1}=u_1,\ldots,S^{[\phi]}_n=u_n\right\}&=&\left\{\Phi_{0}=u_{0},\Phi_{1}=u_1-u_0,\ldots,\Phi_{n}=u_n-u_{n-1}\right\}\\
 &=&   \bigcup_{\substack{s_1,\ldots,s_n: \\ \phi(s_i)=u_i-u_{i-1},1\leq i \leq n\,}} \left\{X_{1}=s_1,\ldots,X_n=s_n\right\}=:A_{n},\\
 \left\{\Phi_{n+1}=k\right\}&=&\bigcup_{\substack{v: \\ \phi(v)=k\,}} \left\{X_{n+1}=v \right\}.
\end{eqnarray*}
Consequently,
    \begin{eqnarray*} \mathbb{P}\left[J_{n+1}=j,\,S^{[\phi]}_{n+1}-S^{[\phi]}_{n}=k\mid J_{0:n},\, S^{[\phi]}_{0:n}=u_{0:n}\right]  &=&      \mathbb{P}\left[J_{n+1}=j,\,\Phi_{n+1}=k\mid J_{0:n},\, S^{[\phi]}_{0:n}=u_{0:n}\right]\\&=&   \sum_{\substack{\substack{v: \\ \phi(v)=k\,}}}\mathbb{P}\left[J_{n+1}=j,\,X_{n+1}=v \mid J_{0:n},\,A_n\right].
    \end{eqnarray*}
    Using the Markov renewal property of the sequence $(J,\,S)$, we conclude that the events \\
    $\left\{J_{n+1}=j,\,X_{n+1}=v \right\}$ and $\{J_{0:n}=i_{0:n},\,A_n\}$  are conditionally independent given the state of $J_n$.
    
    Therefore,
    \begin{equation*}
        \mathbb{P}\left[J_{n+1}=j,\,X_{n+1}=v \mid J_{0:n}=i_{0:n},\,A_n\right]=\mathbb{P}\left[J_{n+1}=j,\,S_{n+1}-S_{n}=v \mid J_{n}=i_n\right]=q_{i_{n}j}(v).
    \end{equation*}
    The corresponding semi-Markov kernel $q^{[\phi]}$ is determined by \eqref{Phi-kernel}.
   % \begin{equation*}
   %     q^{\phi}_{ij}(k)=\sum_{\substack{\substack{v \\  \phi(v)=k\,}}} q_{ij}(v).
    %\end{equation*}
\end{proof}

The following corollary is a direct application of the previous proposition.
\begin{corollary}
    If $(J,\,S)$ is a dt-MRC, then the process $(J,\, S^{[y]})$ with interjump times $Y_{n+1}=X^{[u]}_{n+1}\cdot X^{[v]}_{n+1} $,\, $1\leq u < v\leq d$, forms a Markov renewal chain and its semi-Markov kernel $q^{[y]}$ can be determined by:
    \begin{equation}\label{mult-kernel}
      q^{[y]}_{ij}(k)=\sum_{\substack{\substack{k_{1:d}} \\  k_{u}\cdot k_{v}=k\,}} q_{ij}(k_{1:d}).    
    \end{equation}
\end{corollary}

\begin{remark} For the MRC $(J,S^{[y]})$ the marginal survival function $\overline{H}^{[y]}_{i}(\cdot)$ of state i, is given by
\begin{eqnarray*}
    \overline{H}^{[y]}_{i}(k)&=& \sum_{j}\sum_{l=k+1}^{\infty}\mathbb{P}\left[J_{n+1}=j,\,Y_{n+1}=l \mid J_{n}=i\right] = \sum_{j}\sum_{l=k+1}^{\infty}q^{[y]}_{ij}(l)=\sum_{j} \overline{Q_{ij}^{[y]}}(k).
\end{eqnarray*}
    Thus, for the conditional covariance between $X^{[u]}_{n+1}$ and $X^{[v]}_{n+1}$ we have
    \begin{equation} \label{cond-cov}
        c_{i}^{[u,v]}=\sum_{k=0}^{\infty}\sum_{j}\overline{Q_{ij}^{[y]}}(k)-\left(\sum_{k=0}^{\infty}\sum_{j}\overline{Q_{ij}^{[u]}}(k)\right)\cdot \left(\sum_{k=0}^{\infty}\overline{Q_{ij}^{[v]}}(k)\right).
    \end{equation}
\end{remark}
In the rest of this section we present the extension of the time unidimensional semi-Markov chains to multiple time dimensions and study some of its probabilistic characteristics. The latter is facilitated by a natural extension of the convolution-based solutions of Markov renewal equations from the unidimensional time to the multidimensional case. 

In the development of the theory a fundamental role will be played by the sequence of matrices $\mathbbm{I}_s -q$ and in the following proposition we give the form of its convolutional inverse.
\begin{proposition}
The convolutional inverse of the matrix-valued function $\mathbbm{I}_s\!-\!q$ exists and is given by
\begin{equation}\label{inv-E0-q}
    u:=(\mathbbm{I}_s-q)^{(-1)}=\sum_{n\geq 0}^{}q^{(n)}.
\end{equation}
\end{proposition}
\begin{proof}Since $q(0_d)$ is assumed to be the null matrix then $[\mathbbm{I}_s-q](0_d)=I_s$  and consequently the convolutional inverse of the sequence of matrices $\mathbbm{I}_s-q$ exists by Theorem \ref{algebraicp1}. Furthermore,
by (\ref{form-matrix-inv-unit}) we get directly that
\begin{equation*}
    (\mathbbm{I}_s-q)^{(-1)}=\sum_{n\geq 0}^{}\left(\mathbbm{I}_s-(\mathbbm{I}_s-q)\right)^{(n)}=\sum_{n\geq 0}^{}q^{(n)}.
\end{equation*}
\end{proof}
Relation (\ref{inv-E0-q}) can be rewritten trivially: $u=\mathbbm{I}_s + q*u$. This representation leads us to  a special case of a  multivariate extension of the well-known class of discrete time Markov renewal equations.
\begin{definition}[Discrete multi-time Markov renewal equation] 
Let $L \in $ $\mathcal{M}_{s}(\mathbb{N}^{d})$ be an unknown matrix-valued function and $G \in \mathcal{M}_{s}(\mathbb{N}^{d})$ be a known one. The equation 
\begin{equation}\label{mt-mre}
    L=G + q*L
\end{equation}
is called the discrete multi-time Markov renewal equation (MRE).
\end{definition}
 By (\ref{inv-E0-q}) we get directly the solution to the above equation. 
\begin{corollary}
The multi-time Markov renewal equation (\ref{mt-mre}) has a unique solution given by
\begin{equation*}
    L=u*G.
\end{equation*}
\end{corollary}
The above solution will facilitate the development of the multi-time semi-Markov chain theory similarly to the unidimensional case. Some of its characteristics are defined below.
\begin{definition}\label{tb-SMC}
Let $(J,S)$ be a $d$-dimensional mt-MRC.
The chain $Z=(Z_{k_{1:d}})_{k_{1:d}\in \mathbb{N}^{d}}$ with
 \begin{equation*}
    Z_{k_{1:d}} = J_{N(k_{1:d})},
		\end{equation*} 
is said to be the $d$-dimensional multi-time semi-Markov chain (dt-SMC) associated to $(J,S)$.
\end{definition}
\begin{definition}
The transition function of the dt-SMC $Z$ is denoted by $P$ $\in\mathcal{M}_{s}(\mathbb{N}^{d})$, where
\begin{equation*}
    P_{ij}(k_{1:d}):= \mathbb{P}(Z_{k_{1:d}}=j \mid Z_{0_d}= i ) , \ i,j \in E,\ k_{1:d} \in \mathbb{N}^{d}.
\end{equation*}
\end{definition}
\begin{remark}
For $d=1$, the mt-SMC $Z_{k_{1:d}}$ corresponds to a usual semi-Markov chain.
\end{remark}
\begin{remark} In the theory of the classical semi-Markov processes, any sojourn time is strictly positive (with probability 1). However, in this context zero-time events are allowed marginally and consequently the   marginal process $(Z^{[u]}_{k_u}=J_{N^{[u]}(k_u)})$ isn't necessarily a usual semi-Markov chain.
\end{remark}

Similarly to the time unidimensional case, the transition function of the multi-time semi-Markov chain satisfies a multi-time Markov renewal equation which can be solved easily.
\begin{proposition} The transition function $P\in \mathcal{M}_{s}(\mathbb{N}^{d})$ can be computed as follows:
  \begin{equation} \label{SMC-renewal}
    P= u*\widetilde{H} , 
\end{equation}
where $\widetilde{H}$ is given in Definition
\ref{mt-MRC}-(iii).
%\begin{equation*}
%    \widetilde{H}=\mathrm{diag}\big\{\widetilde{H}_{i}(k_{1:d}) \big\} \in  \mathcal{M}_{s}(\mathbb{N}^{d}).
%\end{equation*}
\end{proposition} 
\begin{proof}
Let us consider the event $A_{k_{1:d}}:=\left\{S_{1}\leq k_{1:d}\right\}$. The transition probability $P_{ij}(k_{1:d})$ can be decomposed using the complementary events $A_{k_{1:d}}$ and $A^{c}_{k_{1:d}}$. The rest of the  proof  is similar to the proof of  Proposition (3.2) in \cite{barbu2009semi}.
\end{proof}
 \begin{definition} Let  $(J,S)$ be a multi-time Markov renewal chain. We denote by  $S^{(j)}:=(S_{n}^{(j)})_n$ the successive passage times in a fixed state $j\in E$, i.e.
  \begin{align*}
   S_{0}^{(j)}&=S_{m},\,\textrm{where}\ \,     m = \min\{ l \in \mathbb{N} \,|\, J_{l}=j \},&\\
      S_{n}^{(j)}&=S_{m},\,\textrm{where}\ \,      m = \min\{ l \in \mathbb{N}\, |\, J_{l}=j,\,S_{l} > S_{n-1}^{(j)} \}.&
  \end{align*}
  \end{definition}
From the above definition, we get directly that $S^{(j)}$ is a simple multi-time renewal chain (mt-RC) on the event  $\{J_{0}=j\}$. Since in general, $S^{(j)}$ is a delayed mt-RC, the chain $S^{(j)}-S_{0}^{(j)}$ plays the role of the associated simple multi-time renewal chain. 
Additionally, $S_{n}^{(j,u)}$ will refer to the $u$-th time component of   $S_{n}^{(j)}$.

For the rest, let us also define the counting process of the number of visits to a state $i\in E$ for the embedded Markov chain $J$ up to time $k_{1:d}$:   
\begin{equation*}
    \widetilde{N}_{i}(k_{1:d}) := \displaystyle  \sum_{n=0}^{N(k_{1:d})} \mathbbm{1}_{\{J_{n}=i\}} = \displaystyle  \sum_{n=0}^{k_{1}+\ldots + k_d} \mathbbm{1}_{\{J_{n}=i,\, S_{n} \leq k_{1:d} \}}
\end{equation*}
and its associated marginal counting processes:
\begin{equation*}
    \widetilde{N}^{[u]}_{i}(k_{u}) := \displaystyle  \sum_{n=0}^{N^{[u]}(k_{u})} \mathbbm{1}_{\{J_{n}=i\}} = \displaystyle  \sum_{n=0}^{k_{u}} \mathbbm{1}_{\{J_{n}=i,\, S^{[u]}_{n} \leq k_{u} \}}, \quad 1 \leq u\leq d.
\end{equation*}
Similarly to \eqref{N_as_min}, we have
\begin{equation}\label{tilde-tilde}
    \widetilde{N}_{i}(k_{1:d}) = \min_{u}\left\{\widetilde{N}^{[u]}_{i}(k_{u})\right\}.
\end{equation}
Also, we notice that the process $\widetilde{N}_{i}-1$ is the counting process associated to the delayed  renewal chain $S^{(i)}$, recording the number of visits to state $i$. Similarly, $\widetilde{N}^{[u]}_{i}-1$ is the counting process associated to the time uni-dimensional delayed renewal chain $S^{(i,u)}$.

       The following definition extends the Markov renewal function to the time multidimensional case. 
\begin{definition} The  matrix sequence $U\in \mathcal{M}_{s}(\mathbb{N}^{d}) $ given by
 \begin{equation*}
    U_{ij}(k_{1:d}) := \mathbb{E}_{i} \left[\widetilde{N}_{j}(k_{1:d})\right], \quad i,j  \in  E,\ k_{1:d} \in \ \mathbb{N}^{d},
\end{equation*}
is called the time $d$-dimensional \textit{Markov renewal function} of the dt-MRC $(J,S)$. 
\end{definition}
Showing that $U$ satisfies a multi-time Markov renewal equation, we get a simplified representation.
 \begin{proposition}\label{markov-function2}
  The multi-time Markov renewal function $U$  satisfies the multi-time MRE:
  \begin{equation*}
      U=\mathbbm{I}_s + q* U,
  \end{equation*}
  and can thus be represented as 
  \begin{equation*}
    U =   \mathbbm{I}_s*u.  
  \end{equation*}
  \end{proposition}

 \begin{definition}\label{quantities-SMC}For any states  $i,j\in E$ we consider
  \begin{itemize} 
     \item[(i)] the joint probability mass function $g_{ij}$ of the first hitting time of state $j$, starting from state $i$ : 
      \begin{equation*}
          g_{ij}(k_{1:d}):= \mathbb{P}_{i}\left(S_{0}^{(j)}=k_{1:d} \right). 
      \end{equation*}
    If $j=i$, then $g_{ii}$ corresponds to the probability mass function of the recurrence time of state $i$, 
      \item[(ii)] the joint cumulative distribution of the first hitting time of state $j$, starting from state $i$ : 
      \begin{equation*}
          G_{ij} := \mathbbm{1}*g_{ij}.
      \end{equation*}
     If $j=i$, then $G_{ii}$ corresponds to the joint cumulative distribution of the recurrence time of $i$,
      \item[(iii)] the survival function of the first hitting time of state $j$, starting from state $i$ :
      \begin{equation*}
          \overline{G}_{ij}(k_{1:d}) := \mathbb{P}_{i}\left(S_{0}^{(j)}>k_{1:d} \right).
      \end{equation*}
      If $j=i$, then $\overline{G}_{ii}$ corresponds to the
      survival function of the recurrence time of state $i$,
     \item[(iv)] the \textit{mean first passage time} $\mu_{ij}=(\mu_{ij}^{[u]})_{1\leq u\leq d}$ from state $i$ to state $j$, where
     \begin{equation*}
      \mu_{ij}^{(u)}:=\mathbb{E}_{i}(S_{0}^{(j,u)}).
      \end{equation*}
      For $j=i$, the vector $\mu_{ii}$ is referred to as the \textit{mean recurrence time} of state $i$. The coordinate $\mu_{ii}^{(u)}$ corresponds to the mean recurrence time of state $i$  
       in terms of the 1t-MRC $(J,S^{[u]})$. The same holds true for the mean first passage times.
      \item[(v)]  the mean product of the first passage time components to state $j$ from  state $i$ for the SMC $Z$ : 
      \begin{equation*}
          \mu_{ij}^{(u,v)}:=\mathbb{E}_{i}\left( S_{0}^{(j,u)}\cdot S_{0}^{(j,v)}\right),
      \end{equation*}
      \item[(vi)]  the covariance  of the first passage time components to state $j$ from  state $i$: 
      \begin{equation*}        \gamma_{ij}^{(u,v)}:=\mu_{ij}^{(u,v)}-\mu^{(u)}_{ij}\cdot\mu^{(v)}_{ij}.
      \end{equation*}
  \end{itemize}
  
    \end{definition}

      Let us now define the matrix valued functions $g=(g_{ij})_{i,j\in E}$, $\, \mathrm{dg}(u)=\mathrm{diag}\{u_{jj}\}_{j \in E}$. The latter has a convolutional inverse since $u(0)=I_s$.
    
  \begin{proposition}\label{reneq} The matrix valued functions $u$ and $g$ satisfy the following relation
  \begin{eqnarray}\label{reneq1}
  g=(u-\mathbb{I}_s)*\mathrm{dg}(u)^{(-1)}.
  \end{eqnarray}
  \end{proposition}
  
  \begin{proof}Note that  $u_{jj}$ is a sequence of renewal probabilities of the simple multi-time renewal chain $S^{(j)}$ which records renewals on $j$ with probability mass function $g_{jj}$. Therefore, we have
  \begin{equation*}
        u_{jj}=(\mathbb{I}_{s})_{ij} + g_{jj}*u_{jj}.
  \end{equation*}
  Now, let us consider two distinct states $i,j\in E$. Then, we have a delayed  renewal chain which records the visits on $j$ when the system starts from $i$ with initial distribution  $g_{ij}$. In this case, $u_{ij}$ and $u_{jj}$ are the sequences of renewal probabilities for the delayed  renewal chain and the simple renewal chain respectively. Therefore,  we obtain
  \begin{equation*}
      u_{ij}= g_{ij}*u_{jj}.
  \end{equation*}
  Consequently, from the above we have the following matrix form
  \begin{equation}\label{reneq2}
      u=\mathbb{I}_s+g*\mathrm{dg}(u).
  \end{equation}
  From $(\ref{reneq2})$ we obtain that
  \begin{equation*}
      g*\mathrm{dg}(u)=u-\mathbb{I}_s
  \end{equation*}
  and  we get directly that  (\ref{reneq1}) holds.
  \end{proof}
  From the above proposition,  we get directly  the following corollary.
  \begin{corollary}\label{Rel-mu} The matrix sequence $G=(G_{ij})_{i,j\in E}$ is given by
  \begin{eqnarray*}
  G&=&(U-\mathbbm{I}_s)*\mathrm{dg}(u)^{(-1)}.
  \end{eqnarray*}
  \end{corollary}
 
  \begin{remark}
 We derive from  relation (\ref{reneq2})  that $q$ can be determined by
 \begin{equation*}
      q=\mathbb{I}_s - (\mathbb{I}_s+g*\mathrm{dg}(u))^{(-1)}.
  \end{equation*}
 \end{remark}

In order to compute the moments of $S^{(j)}_0$, we require an alternative relationship between the functions $g$ and $q$, which is used solely for theoretical purposes.
This relation can be derived using an approach similar to the one-dimensional case, as shown in  \cite{HMRC}.
\begin{proposition}
    The  matrix valued functions $g$ and $q$ are related by the following expression:
    \begin{equation} \label{soj}
    g_{ij}= q_{ij}+\sum_{\substack{r\neq j }}q_{ir}* g_{rj}.
\end{equation}
\end{proposition}
\begin{proof}
    Conditioning on the first  state visited by $J$, we obtain:
\begin{eqnarray*}
    \mathbb{P}_i\left[ S^{(j)}_0=k_{1:d}\right]&=&\sum_{\substack{r \\ x_{1:d}}}\mathbb{P}_i\left[ S^{(j)}_0=k_{1:d},\,J_{1}=r,\,X_{1}=x_{1:d}\mid J_0=i\right]\\
    &=&\mathbb{P}_i\left[ X_{1}=k_{1:d},\,J_{1}=j\right]+\sum_{\substack{r\neq j \\ x_{1:d}}}\mathbb{P}_i\left[ S^{(j)}_0=k_{1:d},\,J_{1}=r,\,X_{1}=x_{1:d}\right].
\end{eqnarray*}
By applying  the strong Markov property, we get that for all $r\neq j$:
\begin{equation*}
 \mathbb{P}_i\left[ S^{(j)}_0=k_{1:d},\,J_{1}=r,\,X_{1}=x_{1:d}\right]=  \mathbb{P}_r\left[ S^{(j)}_0=k_{1:d}-x_{1:d}\right].
\end{equation*}

Therefore,  equation (\ref{soj}) holds.
\end{proof}

\begin{remark}
According to  (\ref{soj}), we observe  the following properties:
   \begin{itemize}
    \item [(i)] if  $J_{1}=j$, $S_{0}^{(j)}$ represents the passage time from  state $i$ to state  $j$.
    \item[(ii)] if $J_{1}=r\neq j$, then  $S_{0}^{(j)}$  and $S_{0}^{(j)}+\left(X_{1}\right)_{ir}$ are identically distributed.   
\end{itemize} 
Furthermore, the moments of any function $\phi$ of $S_{0}^{(j)}$ can be determined using equation (\ref{soj}). Specifically:
\begin{eqnarray}\label{fmoments}
    \mathbb{E}_{i}\left[ \phi\left(S^{[j]}\right)\right]=p_{ij}\cdot \mathbb{E}_{ij}\left[ \phi\left(X_1\right)\right]+\sum_{k_{1:d}}\sum_{\substack{r\neq j \\ x_{1:d}}} \phi(k_{1:d})\cdot p_{ir} \cdot f_{ir}(x_{1:d})\cdot g_{rj}(k_{1:d}-x_{1:d})
\end{eqnarray}
where the summation extends over all $k_{1:d}$ and $r\neq j$.
\end{remark}

Next, we define some useful classes of multi-time Markov renewal chains as multidimensional extensions of well-known classes of time-univariate MRCs.
  \begin{definition}\label{def:erg}
  Let $(J,S)$ be a mt-MRC and $(J,S^{[u]})$, $1\leq u \leq d$, the associated marginal Markov renewal chains. The process $(J,S)$ is said to be
  \begin{enumerate} 
      \item[$(i)$]  \textit{recurrent} if  all the marginal Markov renewal chains are recurrent,
      \item[$(ii)$] \textit{positive recurrent} if all marginal-MRCs are positive recurrent,
      \item[$(iii)$]  \textit{$r_{1:d}$--periodic} if the chain $(J,S^{[u]})$ is periodic with period $r_u$, for $1\leq u \leq d$. If additionally, $r_u=r$, for $1\leq u \leq d$, then $(J,S)$ is said to be simply \textrm{$r$--periodic} and an $1$--periodic mt-MRC is also referred to as \textit{aperiodic},
      \item[$(iv)$] \textit{ergodic} if and only if $(J,S)$ is positive recurrent and aperiodic,
      \item[$(v)$] \textit{irreducible} if and only if $J$ is irreducible as a Markov chain.
  \end{enumerate}
  \end{definition}

  For an ergodic  Markov chain $J$, its limit 
distribution is denoted by the stochastic vector $\nu$, which we use to compute the following moments.

\begin{proposition}\label{erg}For an irreducible and ergodic mt-MRC $(J,S)$, we have that
 \begin{eqnarray*}
 \mu^{(u)}_{jj}&=&\frac{\sum_{i=1}^{s}\nu_i\cdot m^{[u]}_i}{\nu_{j}}=:\frac{\overline{m}^{[u]}}{\nu_{i}}\quad \text{and} \quad     \mu_{jj}=\frac{\sum_{i=1}^{s}\nu_i\cdot m_i}{\nu_{j}}=:\frac{\overline{m}}{\nu_{j}},\\
     \mu_{jj}^{(u,v)}&=&\frac{\sum_{i=1}^{s}\nu_i\cdot m_i^{[u,v]}}{\nu_{j}} + \frac{\sum_{i=1}^{s}\sum_{r\neq j} \nu_i\cdot p_{ir}\left[ m^{[u]}_{ir}\cdot \mu^{(v)}_{rj}+m^{[v]}_{ir}\cdot \mu^{(u)}_{rj} \right] }{\nu_{j}}.
 \end{eqnarray*}

 \end{proposition}
  
  \begin{proof} There are different ways to derive the desired relations. The first can be proven via (\ref{fmoments}). For the second one,  we use the double expectation theorem conditioning on the first visited state of $J$, so
  \begin{equation*}
   \mu_{ij}^{(u,v)}=  \mathbb{E}_i \left[S^{(j,u)}_{0}\cdot S^{(j,v)}_{0}\right]=\sum_{r}p_{ir} \cdot \mathbb{E}_i \left[S^{(j,u)}_{0}\cdot S^{(j,v)}_{0}\mid J_{1}=r\right].
  \end{equation*}
  If $r=j$, then $S^{(j)}_0$ represents simply the sojourn time at state $i$ given that $j$ is the next visited state, so
  \begin{equation*}
   \mathbb{E}_i \left[S^{(j,u)}_{0}\cdot S^{(j,v)}_{0}\mid J_{1}=r\right] =\mathbb{E}_i\left[X^{[u]}_{1}\cdot X^{[v]}_{1}\mid J_{1}=j\right]=m_{ij}^{[u,v]} . 
  \end{equation*}
  
  For the case $r\neq j$, we note that:
  \begin{equation*}
      \left(S^{(j,u)}_{0}\cdot S^{(j,v)}_{0}\mid J_{0}=i,\,J_{1}=r\right)\stackrel{d}{=}\left(\left[S^{(j,u)}_{0}+Y^{[u]}\right]\cdot \left[S^{(j,v)}_{0}+Y^{[v]}\right]\mid J_{0}=r\right),
  \end{equation*}
  where  the random vector $\left(Y^{[u]},\,Y^{[v]}\right)$ follows the cdf $F_{ir}$ and is independent of $S^{(j)}_0$.

  Therefore, we have
\begin{equation*}
  \mathbb{E}_i \left[S^{(j,u)}_{0}\cdot S^{(j,v)}_{0}\mid J_{1}=r\right]= \mathbb{E}_r \left[\left(S^{(j,u)}_{0}+Y^{[u]}\right)\cdot \left(S^{(j,v)}_{0}+Y^{[v]}\right)\right],\, r\neq j.
\end{equation*}
Now, notice that
\begin{eqnarray*}
    \mathbb{E}_r \left[Y^{[u]}\cdot \left(S^{(j,v)}_{0}+Y^{[v]}\right)\right]&=&\mathbb{E}_r \left[Y^{[u]}\cdot S^{(j,v)}_{0}\right]+\mathbb{E}\left[Y^{[u]}\cdot Y^{[v]}\right] \\ 
    &=&
    \mathbb{E} \left[Y^{[u]}\right]\cdot \mathbb{E}_r \left[S^{(j,v)}_{0}\right]+m_{ir}^{[u,v]}
    =m_{ir}^{[u]}\cdot \mu_{rj}^{(v)}+m_{ir}^{[u,v]}.
\end{eqnarray*}
A similar argument yields:
\begin{eqnarray*}
 \mathbb{E}_r \left[S^{(j,u)}_{0}\cdot \left(S^{(j,v)}_{0}+Y^{[v]}\right)\right]  
   &=&
   m_{ir}^{[v]}\cdot \mu_{rj}^{(u)}+\mu_{ir}^{(u,v)}.
\end{eqnarray*}
Combining the above results and summing over all possible $r$, the following equality holds:
\begin{eqnarray*}
    \mu_{ij}^{(u,v)}&=&
    p_{ij}\cdot m_{ij}^{[u,v]} +\sum_{r\neq j}p_{ir}\cdot\left[ m^{[v]}_{ir}\cdot \mu^{[u]}_{rj}+m^{[u]}_{ir}\cdot \mu^{[v]}_{rj}\right] +\sum_{r\neq j }p_{ir}\cdot m_{ir}^{[u,v]}+\sum_{r\neq j} p_{ir}\cdot \mu_{rj}^{(u,v)}\\
    &=&m_{i}^{[u,v]}+\sum_{r\neq j}p_{ir}\cdot\left[ m^{[v]}_{ir}\cdot \mu^{(u)}_{rj}+m^{[u]}_{ir}\cdot \mu^{(v)}_{rj}\right] +\sum_{r\neq j} p_{ir}\cdot \mu_{rj}^{(u,v)}.
\end{eqnarray*}
Multiplying both sides by $\nu_{i}$ and summing over $E$, we have
\begin{eqnarray*}
    \sum_{i}\nu_{i}\cdot \mu_{ij}^{(u,v)}&=&\sum_{i}\nu_{i}\cdot m_{i}^{[u,v]} +\sum_{i}\sum_{r\neq j}\nu_{i}\cdot p_{ir}\cdot\left[ m^{[v]}_{ir}\cdot \mu^{(u)}_{rj}+m^{[u]}_{ir}\cdot \mu^{(v)}_{rj}\right]+\sum_{i}\sum_{r\neq j}\nu_{i}\cdot p_{ir}\cdot \mu_{rj}^{(u,v)}\\
   &=&\overline{m}^{[u,v]}+ \sum_{i}\sum_{r\neq j}\nu_{i}\cdot p_{ir}\cdot\left[ m^{[v]}_{ir}\cdot \mu^{(u)}_{rj}+m^{[u]}_{ir}\cdot \mu^{(v)}_{rj}\right] + \sum_{r\neq j}\nu_{r}\cdot \mu_{rj}^{(u,v)}\\
   &=& \overline{m}^{[u,v]}+ \sum_{i}\sum_{r\neq j}\nu_{i}\cdot p_{ir}\cdot\left[ m^{[v]}_{ir}\cdot \mu^{(u)}_{rj}+m^{[u]}_{ir}\cdot \mu^{(v)}_{rj}\right] +\sum_{r}\nu_{r}\cdot \mu_{rj}^{(u,v)}-\nu_{j}\cdot \mu_{jj}^{(u,v)}.
\end{eqnarray*}
Consequently, 
\begin{equation*}
  \mu_{jj}^{(u,v)}=  \frac{\overline{m}^{[u,v]}+ \sum_{i}\sum_{r\neq j}\nu_{i}\cdot p_{ir}\cdot\left[ m^{[v]}_{ir}\cdot \mu^{(u)}_{rj}+m^{[u]}_{ir}\cdot \mu^{(v)}_{rj}\right]}{\nu_{j}}
\end{equation*}
  \end{proof}
  %%%%%%
The following result is a direct consequence of Proposition \ref{erg}.
  \begin{corollary}
      The second moment of any recurrence time  is determined by
      \begin{equation*}         \mu^{(u,u)}_{jj}=\frac{\sum_{i=1}^{s}\nu_i\cdot m^{[u,u]}_i}{\nu_{j}} + 2\cdot \frac{\sum_{i=1}^{s}\sum_{r\neq j} \nu_i\cdot p_{ir}\cdot m_{ir}^{[u]}\cdot \mu_{rj}^{(u)} }{\nu_{j}}.
      \end{equation*}
  \end{corollary}

 \section{Applications}
In this section, we present two applications of the proposed methodology. The first focuses on a performance comparison between FFT-based and direct convolution methods. The second application involves the analysis of a parametric two-dimensional Markov Renewal Chain, where we compute several quantities of interest.
  \subsection{ Performing Convolution }
We conducted several comparisons between FFT-based and direct convolution methods applied to matrix-valued sequences in both one and two time dimensions. All numerical experiments were performed using RStudio on a laptop with the following specifications: AMD Ryzen 7 5800H with Radeon Graphics and 16GB of memory.

In the first two tables, we report results for the 2-fold convolution and the computation of the convolutional inverse for the matrix sequence $\left[\mathbb{I}_3-q\right] \in \mathcal{M}_{3}\left(\mathbb{N}\right)$, evaluated up to time 512. Each method was tested over 100 independent replications.

For the 2-fold convolution, the FFT-based method proved significantly faster, requiring only 0.45 seconds compared to 36.34 seconds for the direct method. This corresponds to a speedup factor of approximately 80.75.

Regarding the convolutional inverse, both the Gauss-Jordan and Newton methods outperformed the direct approach by a large margin. Specifically, the Gauss-Jordan method was approximately 42.65 times faster, while the Newton method achieved a speedup factor of about 10.18. Overall, the Gauss-Jordan method was around 4.2 times faster than Newton's method.
\begin{table}[H]
\centering
\begin{tabular}{lccccc}
\hline
\textbf{Test} & \textbf{Replications} & \textbf{Elapsed (s)} & \textbf{Relative} & \textbf{User.self} & \textbf{Sys.self} \\
\hline
 \texttt{Direct}  & 100  &  36.34 &  80.756  &   36.21  &    0.05     \\
\texttt{FFT} & 100 &   0.45 &   1.000 &     0.36 &    0.08 \\

\hline
\end{tabular}
\caption{Benchmark results comparing direct and FFT-based convolution methods for time uni-dimensional matrix sequences.}
\end{table}

\begin{table}[H]
\centering
\begin{tabular}{lccccc}
\hline
\textbf{Test} & \textbf{Replications} & \textbf{Elapsed (s)} & \textbf{Relative} & \textbf{User.self} & \textbf{Sys.self} \\
\hline
 \texttt{Direct} & 100  &36.68 &  42.651   &  36.43  &   0.14\\
\texttt{Newton} &100   & 3.61  &  4.198  &    3.43  &   0.17     \\
\texttt{Gauss-Jordan} &100 &   0.86  &  1.000   &   0.78  &   0.10   \\

\hline
\end{tabular}
\caption{Benchmark results comparing Newton, Gauss-Jordan, and direct methods for convolutional inversion for time uni-dimensional matrix sequences.}
\end{table}

In the following set of experiments, we computed the 2-fold convolution and the convolutional inverse of the matrix sequence $\left[\mathbb{I}_3-q\right] \in \mathcal{M}_{3}\left(\mathbb{N}^{2}\right)$ using $k_1=k_2=32$. Each method was tested in over 100 replications. The user.self and sys.self values confirm a significantly higher CPU time requirement for the direct method. Notably, the FFT-based method was approximately 46 times faster than the direct approach.

For the convolutional inverse, both the Gauss-Jordan and Newton methods again demonstrated clear computational advantages over the direct method. Specifically, the Gauss-Jordan method was approximately 52.44 times faster than the direct approach, and about 5.2 times faster than the Newton method.
\begin{table}[H]
\centering
\begin{tabular}{lccccc}
\hline
\textbf{Test} & \textbf{Replications} & \textbf{Elapsed} & \textbf{Relative} & \textbf{User.self} & \textbf{Sys.self} \\
\hline
\texttt{Direct}  & 100 & 101.43 &  46.315  &   99.70  &   0.77 \\
\texttt{FFT}     & 100   & 2.19 &   1.000  &    2.12  &   0.03  \\

\hline
\end{tabular}
\caption{Benchmark results comparing direct and FFT-based convolution methods for time uni-dimensional matrix sequences.}
\end{table}
\begin{table}[H]
\centering
\begin{tabular}{lccccc}
\hline
\textbf{Test} & \textbf{Replications} & \textbf{Elapsed} & \textbf{Relative} & \textbf{User.self} & \textbf{Sys.self} \\
\hline
\texttt{Direct}     &100 &  89.68 &  52.444 &    89.25 &    0.28   \\
\texttt{Newton}   & 100&    8.97&    5.246 &     8.87 &    0.08  \\
\texttt{Gauss-Jordan} & 100 &   1.71 &   1.000 &     1.54 &    0.17 \\
\hline
\end{tabular}
\caption{Benchmark results comparing Newton, Gauss-Jordan, and direct convolution inverse methods for time 2-dimensional matrix sequences.}
\end{table}

In the final set of experiments, we evaluated the 2-fold convolution and the convolutional inverse of a 3×33×3 matrix sequence using both the FFT-based and direct methods.

For the FFT-based method, we processed two matrix sequences up to time $(128,128)$, whereas for the direct method, the computation was limited to $(32,32)$ due to its higher computational burden.

We observe that, despite the FFT-based method handling a much larger portion of the sequence, it still completed significantly faster than the direct approach.

\begin{table}[H]
\centering
\begin{tabular}{lccccc}
\hline
\textbf{Test} & \textbf{Replications} & \textbf{Elapsed} & \textbf{Relative} & \textbf{User.self} & \textbf{Sys.self} \\
\hline

\texttt{Direct}-up to $(32,32)$     & 100 &  83.09  &  10.68   &  82.71  &   0.241 \\
\texttt{FFT}-up to $(128,128)$  & 100  &  7.78 &    1.00  &    7.63    & 0.10 \\
\hline
\end{tabular}
\caption{Benchmark results comparing direct and FFT-based convolution method for different time 2-dimensional   sizes. }
\end{table}
\begin{table}[H]
\centering
\begin{tabular}{lccccc}
\hline
\textbf{Test} & \textbf{Replications} & \textbf{Elapsed} & \textbf{Relative} & \textbf{User.self} & \textbf{Sys.self} \\
\hline
\texttt{Gauss-Jordan}- up to $(128,128)$&  100  & 38.81 &   1.000  &   35.96  &   2.80  \\
\texttt{Direct}-up to $(32,32)$    & 100&   80.56  &  2.076 &    80.30   &  0.14 \\

\hline
\end{tabular}
\caption{Benchmark results comparing  Gauss-Jordan and direct method for different time 2-dimensional sizes.}
\end{table}
  \subsection{ Moments of MRC}
Let us consider a   three-state 2-t Markov renewal chain $(J,S)$, where  the embedded Markov chain $J$ with state space $\{1,2,3\}$, has the following transition matrix
\begin{equation*}
      P=
  \left( {\begin{array}{ccc}
   0 & \frac{1}{3} & \frac{2}{3}\\
   & & \\
  \frac{1}{5} & 0 & \frac{4}{5} \\
  & & \\
   \frac{3}{4}& \frac{1}{4} & 0
  \end{array} } \right),
  \end{equation*}
  and sojourn times depending only on the  current state $i\, \in \left\{1,2,3\right\}$, specified as follows:
  \begin{itemize}
\item $(X_{n+1}|J_{n}=i,\,J_{n+1}=j)\sim BPoisson(\alpha_i,\beta_i,\,\gamma_{i})$ +(1,1), where BPoisson refers to a bivariate Poisson distribution with parameters $(\alpha_i,\beta_i,\,\gamma_{i})$ \cite{10.2996/kmj/1138846776}.
The corresponding pmf is given by:
\begin{equation*}
    f_{i\bullet}(x_{1:2})=\exp{\left\{-\left(\alpha_i+\beta_i+\gamma_i\right)  \right\}}\frac{\alpha^{x_1}_i}{(x_1)!}\frac{\beta^{x_2}_i}{(x_2)!}\sum_{k=0}^{\min(x_1,x_2)} \binom{x_1}{k}\binom{x_2}{k} k! \left(\frac{\gamma_i}{\alpha_i \cdot \beta_i}\right)^{k}
\end{equation*}
\item First moments of the first sojourn time dimension: $m^{[1]}_{i}$=$\alpha_i+\gamma_{i}+1$.
\item First moments of the second sojourn time dimension:  $m^{[2]}_{i}$=$\beta_i+\gamma_{i}+1$.
\item Covariance term: $c_{i}^{[1,2]}=\gamma_{i}$.
\item Mean product: $m_{i}^{[1,2]}=\gamma_{i}+\left(\alpha_i+\gamma_{i}+1\right)\cdot\left(\beta_i+\gamma_{i}+1\right) $.
  \end{itemize}
  
   Then, the semi-Markov kernel is given by the following matrix-valued function  $q$:
  \begin{equation*}
      q=
  \left( {\begin{array}{ccc}
   0 & \ \frac{1}{3}\cdot f_{12} & \ \ \frac{2}{3}\cdot f_{13}\\
   & & \\
   \frac{1}{5}\cdot f_{21} & \ 0 & \ \ \frac{4}{5}\cdot f_{23} \\
   & & \\
   \frac{3}{4}\cdot f_{31}& \ \ \frac{1}{4}\cdot f_{32} & \ \ 0
  \end{array} } \right),
  \end{equation*}
  where $f_{ij}$ denotes the probability distribution function for the sojourn time when transitioning from state $i$ to state $j$.
The stationary distribution of $J$ is given by the stochastic vector
\begin{equation*}
    \nu=\left(\frac{24}{67},\, \frac{15}{67},\,\frac{28}{67}\right).
\end{equation*}
If we choose the following parameter values:
\begin{itemize}
    \item $\alpha=(2,3,4)$,
    \item $\beta=(1,2,3)$,
    \item $\gamma=(3,2,3)$,
\end{itemize}    then the corresponding first and second moments, along with the covariance terms, are:
\begin{itemize}
    \item For the first time dimension  $$m^{[1]}=(6,6,8),$$ $$m^{[1,1]}=(41,41,71).$$
     \item For the second time dimension $$m^{[2]}=(5,5,7),$$ $$m^{[2,2]}=(29,29,55).$$
     \item The covariance term $$c^{[1,2]}=(3,2,3).$$
\end{itemize}
These values correspond to the state-dependent moments and covariances of the sojourn time distributions.

Then, using these inputs, the moments of $S^{(j)}_{0}$, whether joint or marginal, can be calculated.
 Specifically, for the first dimension, the marginal first moments are:
\begin{itemize}
    \item $\mu^{(1)}_{jj}=(19.083, \ 30.533,\ 16.357)$
        \item $\mu^{(1,1)}_{jj}=( 450.042,\ 1361.267,\  299.291)$
\end{itemize}

while for the second dimension, the marginal first moments are:
\begin{itemize}
    \item $\mu^{(2)}_{jj}=(16.292, \ 26.067,\ 13.96429)$
        \item $\mu^{(2,2)}_{jj}=(370.500,\ 1139.133,\  238.046)$
\end{itemize}
Consequently, the variances are:
\begin{itemize}
    \item $Var_{j}\left[S^{(j)}_{0}\right]=(85.868,\ 428.982,\  31.735)$
        \item $Var_{j}\left[S^{(j)}_{0}\right]=(105.082,\ 459.662,\  43.045)$,
\end{itemize}
while the mean product of the corresponding recurrence times  in each state is:        $$\mu^{(1,2)}_{jj}=(378.188,\ 1153.600,\  249.867).$$

Therefore, the corresponding covariances are:
$$c_{jj}=( 67.288,\ 357.698,  \ 21.452),$$
and finally, the desired correlations are
$$ Cor_{j}\left[S^{(j,1)}_{0},S^{(j,2)}_{0}\right]=(0.708,\ 0.806,\ 0.580).$$

  \section{Discussion}
  In this work, we introduced a novel type of Markov renewal chains based on a  time-multidimensional extension of the classical ones. We presented a concise algebraic framework for analyzing their probabilistic structure, focusing on Markov renewal equations and convolution properties. This extension generalizes classical Markov renewal and semi-Markov theory involving multiple time dimensions and offers a clear formulation and faster computations using matrix and convolutional operations. In particular, the convolutional inverse of a matrix sequence, essential for solving Markov renewal equations, is performed using a convolution-based Gauss-Jordan algorithm.

Several directions are open for future exploration. A natural extension of the present framework would be the development of a continuous-time analogue and its approximation by the current study, using discretization methods. Moreover, studying the asymptotic properties of suitable functionals delves into the long-run behavior of a multi-time MRC. From a statistical perspective, both non-parametric and parametric methods could be constructed for estimation purposes. Potential applications include multi-time reliability models.

\begin{appendix}

\section{Fast Fourier Transform }
We begin by introducing the following notations:    \begin{itemize}
        \item $\boldsymbol{i}$: the imaginary unit
        \item $e^{\boldsymbol{i}\,\omega_{1:d}}$= $(e^{\boldsymbol{i}\,\omega_{1}},\ldots,e^{\boldsymbol{i}\,\omega_{d}}) $
        \item $k_{1:d}\cdot \omega_{1:d} $ denotes the inner product $\sum_{i} k_{i} \cdot \omega_{i}$
    \end{itemize}
    \begin{definition}
        The \textit{discrete time Fourier transform} (DFT) of a sequence  $A$ $ \in \mathcal{M}_{s}(\mathbb{N}^{d})$ is a complex matrix valued function defined by
         \begin{equation*}
        \widehat{A}(\omega_{1:d})=\sum_{k_{1:d}}A(k_{1:d})\exp\{-\boldsymbol{i} \left(k_{1:d}\cdot \omega_{1:d}\right)\},\quad\omega_{1:d} \in \mathbb{R}^{d},
    \end{equation*} 
       or, equivalently, $\widehat{A}=\left( \widehat{A}_{ij}(\omega_{1:d})\right)_{i,j \in E}$,
   % \begin{equation*}
    %    \widehat{A}(\omega_{1:d})=\left(\sum_{k_{1:d}} a_{ij}(k_{1:d})\exp\{-\boldsymbol{i} \left(k_{1:d}\cdot \omega_{1:d}\right)\}\right)_{i,j \in E}.
    %\end{equation*}
    %
    where $\widehat{A}_{ij}(\omega_{1:d})$ is the DFT of the real sequence $A_{ij}(\omega_{1:d})$. 
    \end{definition}
    
     %Any  sequence $\widehat{A}$ whose   matrix  elements are Fourier transforms of real sequences, so $$
   % \widehat{A}=\left( \widehat{a}_{ij}(\omega_{1:d})\right)_{i,j \in E}.$$
   
\begin{definition}  Let $\widehat{A}$ be the DFT of a sequence $A$ $ \in \mathcal{M}_{s}(\mathbb{N}^{d})$. The \textit{inverse DFT} is determined by  
     \begin{equation*}
        A_{ij}(k_{1:d})= \frac{1}{(2\pi)^{d}}\int_{[-\pi,\, \pi]^{d}} \widehat{A}_{ij}\left(e^ {\boldsymbol{i}\,\omega_{1:d}} \right) \exp\{\boldsymbol{i} \left(k_{1:d}\cdot \omega_{1:d}\right)\}d\omega_{1:d},\quad k_{1:d} \in \mathbb{N}^{d},\, i,j\in E.
    \end{equation*}
\end{definition}
   
  The convolution theorem for matrix sequences follows from the real-valued sequences one: 
  %from the component-wise convolution property:
  %
    \begin{equation*}
   A*B=C \ \ \Longleftrightarrow \ \ \widehat{A}\cdot \widehat{B}=\widehat{C}.  
    \end{equation*}
     
    For a finite sample sequence, truncated to lie in the finite domain dimensions, $k_{i}$ in the $i$-th dimension, $i=1,\ldots,d$,  the Fourier transform reduces to the multidimensional DFT, which involves the roots of unity in each coordinate direction:
   \begin{equation}\label{DFT}
        \widehat{A}_{ij}(l_{1:d})=\sum_{m_{1:d} \leq k_{1:d}}A_{ij}(m_{1:d})\exp\left\{-2\cdot \pi\cdot \boldsymbol{i} \sum_{r=1}^{d} \frac{l_{r}\cdot m_{r} }{k_{r}} \right\},\quad l_{1:d}\leq k_{1:d}.
    \end{equation}
   Finally, the inverse Fourier transform is applied by:    
 \begin{equation}\label{IDFT}
  A_{ij}(l_{1:d})  = \frac{1}{k_1\cdots k_d } \cdot \sum_{m_{1:d} \leq k_{1:d}}\widehat{A}_{ij}(m_{1:d})\exp\left\{2\cdot \pi\cdot \boldsymbol{i} \sum_{r=1}^{d} \frac{l_{r}\cdot m_{r} }{k_{r}} \right\},\quad l_{1:d}\leq k_{1:d}.
 \end{equation}
   A common technique to accelerate convolution is to employ the circular convolution theorem, which enables the efficient computation of linear convolutions using the FFT.
Let  $k_{1:d},\,N_{1:d} \in \mathbb{Z}^{d}$ be two integer vectors. We then define a component-wise modular shift operation as:
$$k_{1:d} \mod  N_{1:d}=\left(k_{l} \mod N_{l}\right)_{1\leq l \leq d}.$$
\begin{definition}
    The circular convolution of two sequences $A$ and $B$,  denoted by  $A*_{N} B$,   in a $d$-dimensional time domain, is given by
\begin{equation}\label{Circ}
  \left[A*_{N} B\right](k_{1:d})=\sum_{l+l'\leq k_{1:d}} A(l) B(l'\mod N_{1:d} ).
\end{equation}
\end{definition}
If two sequences $A$ and $B$  of length $N_i$ and  $M_i$ in the $i$-th dimension, respectively, are zero padded to length $L_i$ (adding zeros), such that $L_i\geq N_i+M_i-1$, then their circular convolution coincides with the standard linear convolution.  In this case, the linear convolution can be computed using the following steps:
\begin{enumerate}
    \item  Pad both $A$ and $B$ with zero matrices.
    \item Perform a circular convolution on the padded domain.
    \item Extract the final result over the valid index range.
\end{enumerate}
Moreover, the DFT of a circular convolution corresponds to the matrix product of the individual DFTs.
The DFT inherently assumes that the sequences are of finite length.
Thus, when computing convolution via the DFT, the default operation performed is circular convolution, not linear convolution.

Circular convolutions of matrix-valued sequences can be computed far more efficiently by leveraging the Fast Fourier Transform (FFT).
The FFT dramatically reduces computational complexity, especially for high-dimensional or large-scale data  (using algorithms such as \cite{d3ea2d52-5ab2-3128-8b80-efb85267295d}),  making it particularly suitable for large matrix-valued sequences.
 In order to compute the convolution of two matrix-valued sequences $A$ and $B$ using FFT with zero-padding, the following steps are performed:
 \begin{enumerate}
     \item Compute the Fourier  transforms $\widehat{A}$ and $\widehat{B}$.
     \item Compute the matrix product $\widehat{A}\cdot \widehat{B}$.
     \item Determine the inverse DFT of $\widehat{A}\cdot \widehat{B}$. 
     \item Take the convolution $A* B$ as the final output.
 \end{enumerate}

As a result, the overall computational complexity for performing the convolution is reduced to $\mathcal{O}\left(s^{3} \cdot \prod_{i=1}^{d}k_{i}\cdot \log_{2}(k_i) \right)$.
\end{appendix}

\section{Newton's inversion method proof}  \label{1-2-3}
We define the error term $$e_{N}:=1-a\cdot b_{N} \in \left<x_{1},\ldots, x_d\right>^{2^N} ,$$
since $a\cdot b_{N}=1 \mod \left<x_{1},\ldots, x_d\right>^{2^N}$.  Now, if $y:=b_{N}+b_{N}\cdot \alpha$, then we have
\begin{eqnarray*}
y \cdot a 
&=& \left(  b_{N} +  b_{N} \cdot  e_{N} \right) \cdot a = b_{N} \cdot a +  b_{N} \cdot  e_{N} \cdot a \\
&=& 1 -  e_{N} +  b_{N} \cdot  e_{N} \cdot a = 1 -  e_{N} \cdot (1 -  b_{N} \cdot a) \\
&=& 1 -  e_{N}^2.
\end{eqnarray*}

Since  $e_{N}\in \left<x_{1},\ldots, x_d\right>^{2^N} $, it follows that 
$e^{2}_{N} \in \left<x_{1},\ldots, x_d\right>^{2^{N+1}}.$
Hence, $$y \cdot a = 1 \mod \left<x_{1},\ldots, x_d\right>^{2^{N+1}},$$ as required.
From the uniqueness of the power series inverse, we conclude that $y$ contains the elements of the convolutional inverse up to $(2^{N+1},\ldots, 2^{N+1})$. 
\end{document}